\documentclass[a4paper,11pt,reqno]{amsart}
\usepackage{amsthm}
\usepackage{amsmath}
\usepackage{enumerate}
\usepackage{amsfonts}
\usepackage{amssymb}
\usepackage{fullpage}
\usepackage{amsmath,amscd}
\usepackage{graphicx}
\usepackage{array}
\usepackage{amsmath,amssymb,amsfonts}
\usepackage[utf8]{inputenc} 
\usepackage[T1]{fontenc}
\usepackage[english,french]{babel}

   \usepackage{gfsartemisia}

 \usepackage{xspace}
\usepackage{xparse}
\usepackage{graphicx}
\usepackage{float}
\usepackage{pdfpages}
\usepackage[a4paper, margin = 3cm, bottom = 3cm]{geometry}
\usepackage{ifpdf}
\usepackage{marginnote}
\usepackage{hyperref}
\usepackage{tikz}
\usepackage{csquotes}
\usetikzlibrary{calc,matrix,arrows,shapes,decorations.pathmorphing,decorations.markings,decorations.pathreplacing}
\hypersetup{pdfborder=0 0 0, 
	    colorlinks=true,
	    citecolor=black,
	    linkcolor=blue,
	    urlcolor=black,
	    pdfauthor={Guillaume Tahar,Quentin Gendron}
	   }

\usepackage[
	    style=alphabetic,
            isbn=false,
            doi=false,
            url=false,
            backend=bibtex, 
            maxnames = 5,
            sorting=nty,
           ]{biblatex}
\defbibheading{bibliography}[\bibname]{%
  \section*{Références}%
}

\DeclareFieldFormat[article]{title}{{\it #1}}
\DeclareFieldFormat{journaltitle}{{\rm #1}}
\renewbibmacro{in:}{%
  \ifentrytype{article}{}{\printtext{\bibstring{in}\intitlepunct}}}
  \bibliography{res_biblio}

\newcommand{\RR}{\mathbf{R}}

\newcommand{\CC}{\mathbf{C}}
\newcommand{\NN}{\mathbf{N}}
\newcommand{\ZZ}{\mathbf{Z}}
\newcommand{\PP}{{\mathbb{P}}}

\newcommand{\pgcd}{{\rm pgcd}}

\newcommand{\ord}{\operatorname{ord}\nolimits}

\renewcommand{\tilde}{\widetilde}

\newcommand{\rot}{{\rm rot}} 
\newcommand{\ind}{{\rm Ind}}

\newcommand\Res{\operatorname{Res}}

\newcommand{\appres}[1][g]{\mathfrak{R}_{#1}}
\DeclareDocumentCommand{\appresk}{ O{g} O{k}}{\mathfrak{R}_{#1}^{#2}}
\newcommand{\espres}[1][g]{\mathcal{R}_{#1}}
\DeclareDocumentCommand{\espresk}{O{g} O{k}}{\mathcal{R}_{#1}^{#2}}

\newcommand{\barmoduli}[1][g]{{\overline{\mathcal M}}_{#1}}

\newcommand{\omoduli}[1][g]{{\Omega\mathcal M}_{#1}}

\newcommand{\Weierstrass}{Weierstra\ss{}}

\def\={\;=\;} 
\newtheorem{thm}{Théorème}[section]
\newtheorem{cor}[thm]{Corollaire}
\newtheorem{prop}[thm]{Proposition}
\newtheorem{lem}[thm]{Lemme}

\theoremstyle{definition}
\newtheorem{defn}[thm]{Définition}
\theoremstyle{remark}
\newtheorem{rem}[thm]{Remarque}
\theoremstyle{definition}

\theoremstyle{definition}
\newtheorem{ex}[thm]{Exemple}
\theoremstyle{definition}

\theoremstyle{definition}

\numberwithin{equation}{section} 
  \usepackage{gfsartemisia}

\title{différentielles abéliennes à singularités prescrites}
\author{Quentin Gendron}
\address[Quentin Gendron]{Centro de Investigación en Matemáticas, Guanajuato, Gto.,  AP 402, CP 36000, México}
\email{quentin.gendron@cimat.mx}
\author{Guillaume Tahar}
\address[Guillaume Tahar]{Institut de Math{\'e}matiques de Jussieu - UMR CNRS 7586}
\curraddr{Faculty of Mathematics and Computer Science, Weizmann Institute of Science,
Rehovot, 7610001, Israel}
\email{tahar.guillaume@weizmann.ac.il}

\date{\today}
\keywords{Abelian differential, Flat surface, Strata, Residue}
\begin{document}

\selectlanguage{english}

\begin{abstract}
The local invariants of a meromorphic Abelian differential on a Riemann surface of genus $g$ are the orders of zeros and poles, and the residues at the poles. The main result of this paper is that with few exceptions, every pattern of orders and residues can be obtain by an Abelian differential. These exceptions are two families in genus zero when the orders of the poles are either all simple or all nonsimple. Moreover, we even show that the pattern can be realized in each connected component of strata. Finally we give consequences of these results in algebraic and flat geometry. The main ingredient of the proof is the flat representation of the Abelian differentials.
\end{abstract}

\selectlanguage{french}

% \begin{abstract}
% Les invariants locaux d'une différentielle abélienne méromorphe sur une surface de Riemann de genre $g$ sont les ordres des zéros et des pôles, et les résidus aux pôles. Le résultat principal de cet article est qu'à quelques exceptions près, chaque configuration d'ordres et de résidus peut être obtenue par une différentielle abélienne. Ces exceptions sont deux familles en genre zéro où les ordres des pôles sont soit tous simples, soit tous non simples. De plus, nous montrons que chaque configuration de résidus peut être réalisée dans chaque composante connexe des strates. Enfin, nous donnons les conséquences de ces résultats en géométrie algébrique et plate. L'ingrédient principal de la preuve est la représentation plate des différentielles abéliennes.
% \end{abstract}

\maketitle
\setcounter{tocdepth}{1}
\tableofcontents

\section{Introduction}

Soient $X$ une surface de Riemann compacte de genre $g$ et $K_{X}$ son fibré en droites canonique. Une {\em différentielle abélienne} de $X$ est une section méromorphe non nulle de $K_{X}$. Localement, elle s'écrit $f(z)dz$, où $f$ est une fonction méromorphe.
Il est bien connu (voir par exemple \cite[Encadré~III.2]{dSG}) que les invariants en un point~$P$ d'une différentielle abélienne~$\omega$ sont l'{\em ordre} de la différentielle en~$P$ et le {\em résidu $\Res_{P}(\omega)$}  dans le cas où~$P$ est un pôle de $\omega$. 

Ces invariants ne sont pas arbitraires mais vérifient certaines propriétés. Tout d'abord le résidu d'un pôle simple est toujours non nul.  Ensuite, la somme des ordres des zéros et des pôles d'une différentielle est égale à $2g-2$. Enfin la somme des résidus  s'annule. Réciproquement, on peut se demander s'il existe des différentielles méromorphes possédant des invariants locaux prescris satisfaisant ces relations.
Deux théorèmes célèbres apportent des réponses partielles à ce problème. 
\par
Le premier est le théorème de Riemann-Roch (voir le théorème 9.3 de \cite{reyssat}) qui permet d'obtenir certaines différentielles avec des ordres fixés. Toutefois, ce théorème ne donne aucune information sur les résidus de ces différentielles. De plus, ce résultat ne permet de montrer qu'il existe des différentielles pour tous les ordres que nous pouvons considérer. Par exemple, le fait qu'il existe des différentielles holomorphes dont les ordres sont donnés par une partition quelconque de $2g-2$ à été montré dans \cite{masm} par des méthodes complètement différentes.
\par
Le second résultat est le théorème de Mittag-Leffler (voir la proposition 9.3 de \cite{reyssat}) qui démontre l'existence de différentielles avec l'ordre des pôles et les résidus correspondants imposés. Toutefois, ce théorème ne donne aucune information sur les zéros de ces différentielles.

Dans cet article, nous nous proposons de répondre à la question suivante.
\begin{center}
{\em \'Etant donnés les ordres des zéros et pôles ainsi que les résidus aux pôles,\\ existe-t-il une surface de Riemann avec une différentielle possédant ces invariants locaux?}
\end{center}

\smallskip
\par
\subsection{Définitions et résultats principaux}
Afin de préciser la question précédente, nous introduisons un certain nombre de notions. Nous désignerons par 
$$\mu:=(a_{1},\dots,a_{n};-b_{1},\dots,-b_{p};\underbrace{-1,\dots,-1}_{s})$$ une partition de $2g-2$  où les $a_{i}$ sont des nombres strictement positifs et les $b_{i}$ sont supérieurs ou égaux à $2$. 
 La {\em strate} $\omoduli(\mu)$ paramètre les différentielles de type~$\mu$. Les strates (non vides) de différentielles sont des orbifoldes de dimension $2g-1+n$. 
 Le théorème des résidus implique que les strates $\omoduli(a_{1},\dots,a_{n};-1)$ sont vides. Dans la suite nous considérerons les partitions $\mu$  de $2g-2$ qui ne sont pas de la forme $(a_{1},\dots,a_{n};-1)$. Notons qu'en général, les strates ne sont pas connexes et leurs composantes connexes  ont été classifiées par \cite{Bo} dans le cas des  différentielles méromorphes. 
 
 Soit $\omoduli(\mu)$ une strate abélienne, on définit l'{\em espace résiduel de type $\mu$} par 
 \begin{equation}
\espres(\mu) := \left\{ (r_{1},\dots,r_{p+s})\in \CC^{p}\times(\CC^{\ast})^{s}:\ \sum_{i=1}^{p+s} r_{i}=0 \right\}.
\end{equation}
% Comme le résidu d'une différentielle $\omega$ à un pôle simple est non nul et la somme des résidus de $\omega$ est égale à zéro, 
Cet espace paramètre l'ensemble des configurations de résidus que peut prendre une différentielle de $\omoduli(\mu)$.
Soit $\{P_{i}\}_{i\in\{1,\dots,p+s \}}$ l'ensemble des pôles de $\omega$, l'{\em application résiduelle} est 
\begin{equation}
\appres(\mu) \colon \omoduli(\mu) \to \espres(\mu):\ (X,\omega) \mapsto (\Res_{P_{i}}(\xi))_{i \in \{1,\dots,p+s\} }\,.
\end{equation}
Nous décrivons maintenant l'image de l'application résiduelle de chaque composante connexe des strates de différentielles abéliennes. Tout d'abord dans le cas des différentielles sur les surfaces de Riemann de genre $g\geq1$.
\begin{thm}\label{thm:ggeq1abel}
Soit $\omoduli(\mu)$ une strate de différentielles abéliennes avec $g\geq1$. La restriction à chaque composante connexe de $\omoduli(\mu)$ de l'application résiduelle $\appres(\mu)\colon\omoduli(\mu) \to \espres(\mu)$  est surjective.
\end{thm}

Le cas des différentielles sur la sphère de Riemann est donné dans le résultat le suivant.
\begin{thm}\label{thm:geq0keq1}
Soit $\mu=(a_{1},\dots,a_{n};-b_{1},\dots,-b_{p};-1\dots,-1)$ une partition de $-2$ où $-1$ apparaît~$s$ fois, l'image de l'application résiduelle de la strate $\omoduli[0](\mu)$ est donnée par l'un des cas suivant.
\begin{itemize}
\item[i)]  Si $s=0$ et qu'il existe un indice $i$ tel que
\begin{equation}\label{eq:genrezeroresiduzerofr}
 a_{i}>\sum_{j=1}^{p}b_{j}-(p+1),
\end{equation}
alors l'image de $\appres[0](\mu)$ est $\espres[0](\mu)\setminus\left\{0\right\}$.
\item[ii)] L'image de $\appres[0](a_{1},\dots,a_{n};-1,\dots,-1)$ avec $s\geq2$ est le complémentaire de l'union des plans $\CC^{\ast} \cdot (x_{1},\dots,x_{s_{1}},-y_{1},\dots,-y_{s_{2}})$ où $x_{i},y_{j} \in \NN$ sont premiers entre eux et  
\begin{equation}\label{eq:polessimples}
 \sum_{i=1}^{s_{1}} x_{i} = \sum_{j=1}^{s_{2}} y_{j} \leq \max(a_{1},\dots,a_{n})\,.
\end{equation}
\item[iii)] Dans les autres cas, l'application résiduelle $\appres[0](\mu)$ est surjective.
\end{itemize}
\end{thm}

\smallskip
\par
\subsection{Applications}
\label{sec:apliintro}

Nous donnons deux applications, l'une à la géométrie de l'espace des modules des courbes algébriques pointées et l'autre aux différentielles abéliennes holomorphes. 
\smallskip
\par
\paragraph{\bf Lieu de Weierstra\ss.}
La première conséquence est un résultat sur la géométrie de l'espace des modules des courbes pointées. Notons que la première des deux assertions a été prouvée par Eisenbud et Harris dans \cite[théorème 3.1]{eiha} à l'aide des séries linéaires limites. L'illustration des courbes stables pointées de cet énoncé est donnée dans la figure~\ref{fig:courbesWei}.
\begin{prop}\label{prop:limWei}
Soit $\overline{\mathfrak{W}_{g}}$ l'adhérence dans $\barmoduli[g,1]$ du lieu paramétrant les points de Weierstra\ss{}~des courbes algébriques de genre $g$. 
\begin{itemize}
\item[i)] Le lieu $\overline{\mathfrak{W}_{g}}$ n’intersecte pas le lieu des courbes stables où $g$ courbes elliptiques sont attachées à un $\PP^{1}$ contenant le point marqué.
\item[ii)] Le lieu $\overline{\mathfrak{W}_{g}}$ intersecte le lieu où $g-1$ courbes elliptiques sont attachées à un $\PP^{1}$ contenant le point marqué et l'une de ces courbes elliptiques est attachée par deux points au $\PP^{1}$. 
\end{itemize}
\end{prop}

\smallskip
\par
\paragraph{\bf Cylindres sur une surface plate.}
La seconde application est sur  la géométrie des surfaces de translations associées à une différentielle abélienne holomorphe.  Naveh a montré dans \cite{Na} que le nombre maximal de cylindres disjoints dans une différentielle holomorphe de la strate $\omoduli(a_{1},\dots,a_{n})$  est $g+n-1$ et que cette borne est toujours atteinte. Nous précisons ce résultat dans deux directions. 

Tout d'abord, nous décrivons les périodes possibles des circonférences de ces cylindres. Notons que la période d'un cylindre est un nombre complexe non nul bien définie modulo multiplication par~$\pm 1$.
% Une version équivalente de ce résultat dans le langage des graphes est la proposition~\ref{prop:cylindresgraphe}.
\begin{prop}\label{prop:cylindres}
% Soient $S:=\omoduli(a_{1},\dots,a_{n})$ une strate de différentielles abéliennes holomorphes et $\lambda:=(\lambda_{1},\dots,\lambda_{t})\in (\mathbb{C}^{\ast})^{t}$ avec $t \leq g+n-1$. Il existe  une différentielle dans $S$ avec $t$ cylindres disjoints dont les circonférences sont $\lambda_{1},\dots,\lambda_{t}$ si et seulement si  il n'est pas de la forme suivante. On a $t \geq g$,  $\lambda_{i} = c\cdot \nu_{i}$ avec $\nu_{i} \in \ZZ$ et il existe une partition des $\pm\lambda_{i}$ en $n$ ensembles de tailles $a_{j}+2$ tels que   $\sum \nu_{i} =0$ pour chaque ensemble et $\sum |\nu_{i}| \leq \max (a_{1},\dots,a_{n})$.
 Soient $S:=\omoduli(a_{1},\dots,a_{n})$ une strate de différentielles abéliennes holomorphes et $\lambda:=(\lambda_{1},\dots,\lambda_{t})\in (\mathbb{C}^{\ast}/\{\pm1\})^{t}$.
Il existe  une différentielle dans $S$ avec $t$ cylindres disjoints dont les circonférences sont $\lambda_{1},\dots,\lambda_{t}$ si et seulement s'il existe une
différentielle stable $(X,\omega)$ avec   $t$ nœuds, telle que:
\begin{enumerate}
 \item la restriction de $\omega$ à chaque composante irréductible possède un pôle simple aux points nodaux;
 \item les résidus aux points nodaux correspondant à un nœud sont $\pm \lambda_{i}$;
 \item les singularités de la restriction de $\omega$ à la partie lisse de $X$ sont d'ordres  $(a_{1},\dots,a_{n})$.
\end{enumerate}
\end{prop}
Cette proposition permet de montrer que tous les $t$-uplets sont des circonférences si $t<g$. Toutefois, il existe en général des restrictions si $t=g$. Ces faits sont énoncés dans le corollaire~\ref{cor:cylindres}.

D'autre part nous montrons que la borne sur le nombre de cylindres est atteinte dans chaque composante connexe des strates de différentielles holomorphes.
\begin{prop}\label{prop:cylcc}
 Dans chaque composante connexe de $\omoduli(a_{1},\dots,a_{n})$ il existe une différentielle abélienne avec $g+n-1$ cylindres.
\end{prop}

\smallskip
\par
\subsection{Liens avec d'autres travaux}
\label{sec:travauxintro}

Le présent texte est une version améliorée de la partie sur les différentielles abéliennes du papier \cite{geta}. En particulier, les théorèmes~\ref{thm:ggeq1abel} et~\ref{thm:geq0keq1}  correspondent respectivement à la proposition~1.3 et au théorème~1.5 de \cite{geta}. L'amélioration la plus substantielle à été opérée dans le point ii) du théorème~\ref{thm:geq0keq1}. La partie de \cite{geta} consacrée aux différentielles quadratiques apparaîtra dans \cite{getaquad} et celle  consacrée aux $k$-différentielles avec $k\geq 3$ dans \cite{getakdiff}.  

Jusqu'à présent la question que nous traitons dans cet article a été considérée uniquement dans certains cas particuliers en genre zéro. Quelques résultats se trouvent dans la Section~2 de \cite{eiha} en lien avec les séries linéaires limites. \cite{Erm} donne de manière indépendante un énoncé équivalent au point ii) du théorème~\ref{thm:geq0keq1} en lien avec les métriques sphériques. Un cas simple du point i) du théorème~\ref{thm:geq0keq1} apparaît dans \cite[Lemme 3.6]{BCGGM} en lien avec les dégénérations de différentielles abéliennes.   Plus précisément, l'article \cite{BCGGM} nous permet de comprendre les limites  des différentielles d'une strate. Dans ce cas, on a besoin de connaître l'existence de différentielles dont les ordres des zéros et des pôles ainsi que des résidus ont été fixés. Notre article peut donc être vu comme la dernière pierre dans la description du bord des strates de différentielles abéliennes donnée par~\cite{BCGGM}.

En plus des applications que nous donnons dans la section~\ref{sec:appli}, les résultats de cet article sont utilisés en géométrie tropicale \cite{MUW}, pour étudier la dégénération des points de \Weierstrass~de manière plus précise \cite{genWei} ou encore pour l'étude de certains espaces d'Hurwitz \cite{mulSecond}.

\smallskip
\par
\subsection{Organisation de cet article}

Le schéma de la preuve de ces théorèmes est le suivant. Dans un premier temps nous utilisons la correspondance entre les différentielles méromorphes et certaines classes de surfaces plates introduites par \cite{Bo} et étudiées dans \cite{tahar}. Cette correspondance nous permet de construire explicitement des différentielles ayant les propriétés souhaitées lorsque le genre et le nombre de singularités sont petits. 

Dans un second temps, nous déduisons de  ces résultats les autres cas grâce à deux opérations introduites par \cite{kozo}. Ces deux opérations sont l'{\em éclatement d'un zéro} et la {\em couture d'anse}. La première de ces opérations permet d'augmenter le nombre de singularités sans changer le genre d'une différentielle. La seconde préserve le nombre de singularités mais augmente le genre de la surface de Riemann sous-jacente.

Enfin, dans les cas où l'application résiduelle n'est pas surjective, nous développons des méthodes ad hoc afin de montrer la non-existence de différentielles possédant certains invariants locaux. 
\smallskip
\par
L'article s'organise comme suit. Pour terminer cette introduction, nous posons quelques conventions. Dans la section~\ref{sec:bao} nous faisons les rappels nécessaires sur les représentations plates des différentielles méromorphes et sur les deux opérations précédemment citées. De plus, nous introduisons les briques élémentaires qui nous permettrons de construire les différentielles avec les propriétés souhaitées. La section~\ref{sec:g0} est dédiée au cas du genre $g=0$ et la section~\ref{sec:ggeq1} à celui des genres $g\geq1$.  Enfin les preuves des applications de la section~\ref{sec:apliintro} sont données dans la section~\ref{sec:appli}.

\smallskip
\par
\subsection{Conventions}
Dans cet article nous définirons le {\em résidu} d'une forme différentielle comme étant le résidu usuel multiplié par la constante $2i\pi$. En particulier, le théorème des résidus s'énonce 
\[\int_{\gamma}\omega=\sum_{i} \Res_{P_{i}}\omega \,,\]
où les $P_{i}$ sont les pôles de $\omega$ encerclés par $\gamma$.   Cette convention n'a aucune incidence sur les énoncés des résultats, mais rend les preuves plus agréables.

Si une strate paramètre des différentielles avec $m$ singularités égales à $a$, alors nous noterons cela $(a^{m})$. Par exemple $\omoduli[3](3,3;-1,-1)$ pourra être notée $\omoduli[3]((3^{2});(-1^{2}))$. Plus généralement, si nous considérons une suite $(a,\dots,a)$ de $m$ nombres complexes tous identiques, nous noterons cette suite $(a^{m})$. Nous espérons que ces notations seront claires dans le contexte.

\smallskip
\par
\subsection{Remerciements} 
Nous remercions Corentin Boissy pour des discussions enrichissantes liées à cet article. Une question de David Aulicino est à l'origine de la proposition~\ref{prop:cylindres}. Le premier auteur remercie l'{\em Institut für algebraische Geometrie} de la {\em Leibniz Universität Hannover} et le {\em Centro de Ciencias Matemáticas} de la {\em Universidad Nacional Autonoma de México} où il a élaboré une grande partie de ce texte. Le deuxième auteur est financé par l’Israel Science Foundation (grant No. 1167/17) et le European Research Council (ERC) dans le cadre du European Union Horizon 2020 Research and Innovation Programme (grant agreement No. 802107).
Enfin, nous remercions le rapporteur anonyme d'une première version de cet article pour sa relecture attentive et ses remarques qui ont grandement amélioré la qualité de cet article.

\section{Boîte à outils}
\label{sec:bao}

Dans cette section, nous introduisons les objets et les opérations de base de nos constructions.  Nous introduisons dans la section~\ref{sec:briques} les briques élémentaires pour construire des surfaces plates. Nous poursuivons par un rappel sur les différentielles entrelacées et les opérations d’éclatement d'un zéro et de couture d'anse dans la section~\ref{sec:pluridiffentre}. Enfin, nous discuterons un cas spécial de surfaces plates dans la section~\ref{sec:coeur}. 

\subsection{Briques élémentaires}
\label{sec:briques}

Dans ce paragraphe, nous introduisons des surfaces plates à bord qui nous serviront de briques pour construire les différentielles méromorphes ayant les propriétés locales souhaitées.

Les éléments de base de nos constructions seront une généralisation des {\em domaines basiques} introduit par \cite{Bo}.  Étant donnés des vecteurs $(v_{1},\dots,v_{l})$ dans $(\CC^{\ast})^{l}$. Nous considérons la ligne brisée $L$ dans $\CC$ donnée par la concaténation d'une demi-droite correspondant à $\RR_{-}$, des $v_{i}$ pour $i$ croissant et d'une demi-droite correspondant à $\RR_{+}$.
Nous supposerons que les $v_{i}$ sont tels que $L$ ne possède pas de points d'auto-intersection. 
 
Le {\em domaine basique positif (resp. négatif)} $D^{+}(v_{1},\dots,v_{l})$ (resp. $D^{-}(v_{1},\dots,v_{l})$) est l'adhérence de la composante connexe de $\CC\setminus L$ contenant les nombres complexes au dessus (resp. en dessous) de~$L$.
Étant donné un domaine positif $D^{+}(v_{1},\dots,v_{l})$ et un négatif $D^{-}(w_{1},\dots,w_{l'})$, on construit le {\em domaine basique ouvert à gauche (resp. droite)} $D_{g}(v_{1},\dots,v_{l};w_{1},\dots,w_{l'})$ (resp. $D_{d}(v_{1},\dots,v_{l};w_{1},\dots,w_{l'})$) en collant par translation les deux demi-droites correspondant à~$\RR_{+}$ (resp. $\RR_{-}$).

On se donne $b\geq2$ et $\tau\in\left\{1,\dots,b-1\right\}$.
Soient $(v_{1},\dots,v_{l};w_{1},\dots,w_{l'})$ des vecteurs de~$\CC^{\ast}$ tels que la partie réelle de leurs sommes est positive et que l'argument (pris dans $\left]-\pi,\pi\right]$) des $v_{i}$ est décroissant, des $w_{j}$ est croissant.
La  {\em partie polaire d'ordre $b$ et de type $\tau$ associée à $(v_{1},\dots,v_{l};w_{1},\dots,w_{l'})$} est la surface plate à bord obtenue de la façon suivante. Prenons l'union disjointe de  $\tau-1$ domaines basiques ouverts à gauche associé à la suite vide, $b-\tau-1$   domaines basiques ouverts à droite associé à la suite vide. Enfin prenons le domaine positif associé aux $v_{i}$ et le domaine négatif associé aux $w_{j}$. On colle alors par translation la demi-droite  inférieure du $i$-ième domaine polaire ouvert à gauche à la demi-droite supérieure du $(i+1)$-ième. La demi-droite inférieure du domaine $\tau-1$ est identifiée à la demi droite de gauche du domaine positif. La demi-droite de gauche du domaine négatif est identifiée à la positive du premier domaine ouvert à gauche. On procède de même à droite. La  figure~\ref{fig:ordreplusmoins} illustre cette construction.

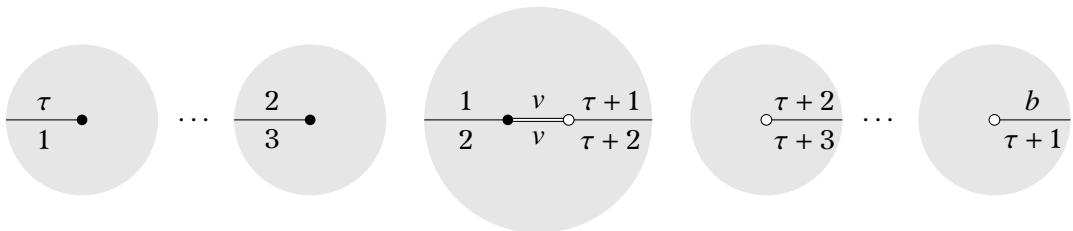
\begin{figure}[htb]
\center
\begin{tikzpicture}[scale=1]
%Figure haut gauche

\begin{scope}[xshift=-6cm]
\fill[fill=black!10] (0,0)  circle (1cm);

\draw[] (0,0) coordinate (Q) -- (-1,0) coordinate[pos=.5](a);

\node[above] at (a) {$\tau$};
\node[below] at (a) {$1$};

\fill (Q)  circle (2pt);

\node at (1.5,0) {$\dots$};
\end{scope}

\begin{scope}[xshift=-3cm]
\fill[fill=black!10] (0,0)  circle (1cm);

\draw[] (0,0) coordinate (Q) -- (-1,0) coordinate[pos=.5](a);

\node[above] at (a) {$2$};
\node[below] at (a) {$3$};

\fill (Q)  circle (2pt);
\end{scope}
%deuxieme figure
\begin{scope}[xshift=0cm]
\fill[fill=black!10] (0,0)  circle (1.5cm);
      \fill[color=white]
      (-.4,0.02) -- (.4,0.02) -- (.4,-0.02) -- (-.4,-0.02) --cycle;

\draw[] (-.4,0.02)  -- (.4,0.02) coordinate[pos=.5](a);
\draw[] (-.4,-0.02)  -- (.4,-0.02) coordinate[pos=.5](b);

\node[above] at (a) {$v$};
\node[below] at (b) {$v$};

\draw[] (-.4,0) coordinate (q1) -- (-1.5,0) coordinate[pos=.5](d);
\draw[] (.4,0) coordinate (q2) -- (1.5,0) coordinate[pos=.5](e);

\node[above] at (d) {$1$};
\node[below] at (d) {$2$};
\node[above] at (e) {$\tau+1$};
\node[below] at (e) {$\tau+2$};

\fill (q1) circle (2pt);
\fill[white] (q2) circle (2pt);
\draw (q2) circle (2pt);

\node at (4.5,0) {$\dots$};
\end{scope}

\begin{scope}[xshift=3cm]
\fill[fill=black!10] (0,0) coordinate (Q) circle (1cm);

\draw[] (0,0) coordinate (Q) -- (1,0) coordinate[pos=.5](a);

\node[above] at (a) {$\tau +2$};
\node[below] at (a) {$\tau+3$};

\fill[white] (Q)  circle (2pt);
\draw (Q)  circle (2pt);
\end{scope}
%troisieme dessin
\begin{scope}[xshift=6cm]
\fill[fill=black!10] (0,0) coordinate (Q) circle (1cm);

\draw[] (0,0) coordinate (Q) -- (1,0) coordinate[pos=.5](a);

\node[above] at (a) {$b$};
\node[below] at (a) {$\tau+1$};

\fill[white] (Q)  circle (2pt);
\draw (Q)  circle (2pt);
\end{scope}

\end{tikzpicture}
\caption{Une partie polaire d'ordre $b$ de type $\tau$ associée à $(v;v)$. Les demi-droites dont les labels coïncident sont identifiés par translation.} \label{fig:ordreplusmoins}
\end{figure}

Si $\sum v_{i} =\sum w_{j}$ nous dirons que cette partie polaire est {\em triviale}. Dans le cas contraire, nous dirons que la partie polaire est {\em non triviale}. Sur la figure~\ref{fig:partiesnontrivial}, le dessin de gauche illustre une partie polaire non triviale.

On se donne maintenant des vecteurs $(v_{1},\dots,v_{l})$ avec $l\geq1$ tels que la concaténation $V$ de ces vecteurs dans cet ordre n'a pas de points d'auto-intersection. De plus, on suppose qu'il existe deux demi-droites parallèles  $L_{D}$ et $L_{F}$ de vecteur directeur $\overrightarrow{l}$, issues respectivement du point de départ $D$ et final $F$ de $V$, ne rencontrant pas $V$ et telles que $(\overrightarrow{DF},\overrightarrow{l})$ est une base positive de $\RR^{2}$. On définit la {\em partie polaire $C(v_{1},\dots,v_{l})$ d'ordre $1$ associé aux $v_{i}$}  comme le quotient du sous-ensemble de $\CC$ entre $V$ et les demi-droites $L_{D}$ et~$L_{F}$ par l'identification de $L_{D}$ à~$L_{F}$ par translation. Le résidu du pôle simple correspondant est donné par la somme $F-D$ des~$v_{i}$. Une partie polaire d'ordre $1$ est donnée à droite de la figure~\ref{fig:partiesnontrivial}.

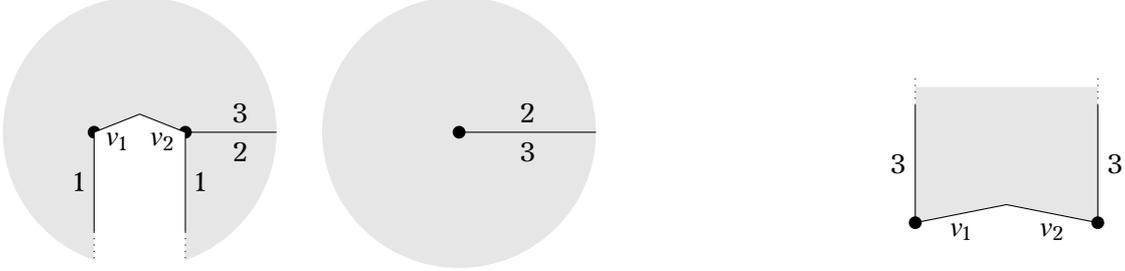
\begin{figure}[htb]
\center
\begin{tikzpicture}[scale=1.2]

%premier dessin
\begin{scope}[xshift=-7cm]
\fill[fill=black!10] (0,0) coordinate (Q) circle (1.5cm);

\coordinate (a) at (-.5,0);
\coordinate (b) at (.5,0);
\coordinate (c) at (0,.2);

\fill (a)  circle (2pt);
\fill[] (b) circle (2pt);
    \fill[white] (a) -- (c)coordinate[pos=.5](f) -- (b)coordinate[pos=.5](g) -- ++(0,-2) --++(-1,0) -- cycle;
 \draw  (a) -- (c) coordinate () -- (b);
 \draw (a) -- ++(0,-1.1) coordinate (d)coordinate[pos=.5] (h);
 \draw (b) -- ++(0,-1.1) coordinate (e)coordinate[pos=.5] (i);
 \draw[dotted] (d) -- ++(0,-.3);
 \draw[dotted] (e) -- ++(0,-.3);
\node[below] at (f) {$v_{1}$};
\node[below] at (g) {$v_{2}$};
\node[left] at (h) {$1$};
\node[right] at (i) {$1$};

\draw (b)-- ++ (1,0)coordinate[pos=.6] (j);
\node[below] at (j) {$2$};
\node[above] at (j) {$3$};
    \end{scope}

%deuxieme figure
\begin{scope}[xshift=-3.5cm]
\fill[fill=black!10] (0,0) coordinate (Q) circle (1.5cm);

\draw[] (0,0) -- (1.5,0) coordinate[pos=.5](a);

\node[above] at (a) {$2$};
\node[below] at (a) {$3$};
\fill[] (Q) circle (2pt);
\end{scope}

%troisieme figure
\begin{scope}[xshift=2.5cm,yshift=-1cm]
\coordinate (a) at (-1,0);
\coordinate (b) at (1,0);
\coordinate (c) at (0,.2);

    \fill[fill=black!10] (a) -- (c)coordinate[pos=.5](f) -- (b)coordinate[pos=.5](g) -- ++(0,1.5) --++(-2,0) -- cycle;
    \fill (a)  circle (2pt);
\fill[] (b) circle (2pt);
 \draw  (a) -- (c) coordinate () -- (b);
 \draw (a) -- ++(0,1.3) coordinate (d)coordinate[pos=.5](h);
 \draw (b) -- ++(0,1.3) coordinate (e)coordinate[pos=.5](i);
 \draw[dotted] (d) -- ++(0,.3);
 \draw[dotted] (e) -- ++(0,.3);
\node[below] at (f) {$v_{1}$};
\node[below] at (g) {$v_{2}$};
\node[left] at (h) {$3$};
\node[right] at (i) {$3$};

    \end{scope}
\end{tikzpicture}
\caption{Une partie polaire non triviale associée à $(v_{1},v_{2};\emptyset)$ d'ordre $3$ (de type $1$) à gauche et d'ordre $1$ à droite.} \label{fig:partiesnontrivial}
\end{figure}

\par
D'après le théorème des résidus, l'intégration de $\omega$ le long d'un lacet fermé $\gamma$ est égale à la somme des résidus aux pôles encerclés par $\gamma$. En effet, rappelons notre convention que notre résidu est égal à $2i\pi$ fois le résidu usuel.
Cela a la conséquence suivante qui, bien qu'élémentaire, est primordiale pour notre étude.
\begin{lem}\label{lm:residu}
Soient $(v_{1},\dots,v_{l};w_{1},\dots,w_{l'})$ des nombres complexes, le pôle associé à la partie polaire d'ordre~$b$ et de type $\tau$ associée à $(v_{1},\dots,v_{l};w_{1},\dots,w_{l'})$ est d'ordre $-b$ et de résidu égal à $\sum_{i=1}^{l} v_{i}-\sum_{j=1}^{l'} w_{j}$. 

Soit $(v_{1},\dots,v_{l})$ avec $l\geq1$, le pôle associé au domaine polaire d'ordre~$1$ associé à $v_{i}$ est d'ordre~$-1$ et possède un résidu égal à $\sum_{i=1}^{l} v_{i}$.
\end{lem}

\subsection{Différentielles entrelacées, éclatement de zéros et couture d'anses}
\label{sec:pluridiffentre}

Dans ce paragraphe, nous décrivons certains cas particuliers des résultats obtenus dans \cite{BCGGM} au sujet des différentielles entrelacées. Cela nous permet de rappeler les constructions de l'{\em éclatement des zéros} et de la {\em couture d'anse}. 

Tout d'abord, nous rappelons la définition d'une différentielle entrelacée.
\'Etant donnée une partition~$\mu:=(m_{1},\dots,m_{t})$ telle que $\sum_{i=1}^t m_i = 2g-2$, une {\em différentielle entrelacée $\eta$ de type~$\mu$}
sur une courbe stable $t$-marquée $(X,z_1,\ldots,z_t)$
est une collection de
différentielles non nulles~$\eta_v$ sur les composantes irréductibles~$X_v$ de~$X$ satisfaisant aux conditions suivantes.
\begin{itemize}
\item[(0)] {\bf (Annulation comme prescrit)} Chaque différentielle $\eta_v$ est méromorphe et le support de son diviseur est inclus dans l'ensemble des nœuds et des points marqués de $X_v$. De plus, si un point marqué $z_i$ se trouve sur~$X_v$, alors $\ord_{z_i} \eta_v=m_i$.
\item[(1)] {\bf (Ordres assortis)} Pour chaque nœud de $X$ qui identifie $q_1 \in X_{v_1}$ à $q_2 \in X_{v_2}$,
$$\ord_{q_1} \eta_{v_1}+\ord_{q_2} \eta_{v_2} = -2\, . $$
\item[(2)] {\bf (Résidus assortis aux pôles simples)} Si à un nœud de $X$
qui identifie $q_1 \in X_{v_1}$ avec $q_2 \in X_{v_2}$ on a $\ord_{q_1}\eta_{v_1}=
\ord_{q_2} \eta_{v_2}=-1$, alors l'équation suivante est vérifiée
$$\Res_{q_1}\eta_{v_1} + \Res_{q_2}\eta_{v_2} = 0 \,.$$
\end{itemize}

Ce n'est que pour des cas très particuliers que nous aurons besoin de savoir quand une différentielle entrelacée est lissable dans la strate $\omoduli(\mu)$. Nous rappelons ici uniquement les cas qui nous intéressent.  Le premier cas est celui où l'ordre des différentielles à tous les points nodaux est égal à $-1$.
\begin{lem}\label{lem:lisspolessimples}
Soit $\eta=\left\{\eta_{v}\right\}$ une différentielle entrelacée de type $\mu$. Si l'ordre des différentielles~$\eta_v$ aux nœuds est $-1$, alors $\eta$ est lissable dans la strate $\omoduli(\mu)$ sans modifier les résidus aux pôles aux points non nodaux.
\end{lem}
 Notons que dans ce cas, la notion de différentielle entrelacée correspond à la notion classique de différentielle stable. Ainsi nous pourrons nous ramener à ce résultat pour prouver le lemme~\ref{lm:g0p-1plusieurszero} et la proposition~\ref{prop:cylindres}.

Maintenant nous regardons le cas des différentielles entrelacées à deux composantes.
\begin{lem}\label{lem:lissdeuxcomp}
 Supposons que $X$ possède exactement deux composantes $X_{1}$ et $X_{2}$ reliées par un unique nœud  qui identifie $q_1 \in X_{1}$ à $q_2 \in X_{2}$. Si  une différentielle entrelacée de type $\mu$ vérifie $\ord_{q_1} \eta_{1}>-1>\ord_{q_2} \eta_{2}$, alors elle est lissable dans $\omoduli(\mu)$ si et seulement si  $\Res_{q_2}\eta_{2}=0$.
 
 De plus, si cette condition est vérifiée, alors le lissage peut se faire sans modifier les résidus aux pôles de $\eta_{1}$.
\end{lem}
Remarquons que le fait que le lissage peut se faire sans modifier les résidus n'est pas explicitement prouvé dans \cite{BCGGM}. Toutefois, cela se déduit  des techniques de \cite{BCGGM} de la façon suivante. Comme le résidu est nul, il n'y a pas besoin de modifier la différentielle sur $X_{1}$ pour pouvoir lisser. Il suffit donc de remplacer le nœud par un cylindre plat. 
\smallskip
\par
Maintenant, nous donnons deux applications cruciales du lemme~\ref{lem:lissdeuxcomp}.
\begin{prop}[Éclatement d'un zéro]\label{prop:eclatZero}
Soient $(X,\omega)$ une différentielle de type $\mu$ et $z_{0}\in X$ un zéro d'ordre $a_{0}$ de $\omega$. Soit $(\alpha_{1},\dots,\alpha_{t})$ un $t$-uplet d'entiers positifs tels que $\sum_{i}\alpha_{i}=a_{0}$. 

Il existe une opération sur $(X,\omega)$ en $z_{0}$ qui fournit une  différentielle $(X',\omega')$ de type $(\alpha_{0},\dots,\alpha_{t},\mu\setminus\lbrace a_{0}\rbrace)$ qui ne modifie pas les résidus aux pôles de $\omega$. 
\end{prop}

\begin{proof}
 Partons de $(X,\omega)$. On forme une différentielle entrelacée en attachant au point~$z_{0}$ une droite projective avec une différentielle ayant les ordres souhaités. Le lemme~\ref{lem:lissdeuxcomp} implique directement la proposition~\ref{prop:eclatZero}.
\end{proof}

La seconde construction nous permettra en particulier de faire une récurrence sur le genre des surfaces de Riemann.

\begin{prop}[Couture d'anse]\label{prop:attachanse}
 Soient $(X,\omega)$ une différentielle abélienne dans la strate $\omoduli(\mu)$ et $z_{0}\in X$ un zéro d'ordre $a_{0}$ de $\omega$. 
 Il existe une opération qui ne modifie pas les résidus aux pôles de $\omega$ et qui donne une différentielle $(X',\omega')$ dans la strate $\omoduli[g+1](a_{0}+2,\mu\setminus\left\{a_{0}\right\})$. 
\end{prop}
\begin{proof}
Partons de $(X,\omega)$. On forme une différentielle entrelacée en attachant au point $z_{0}$ une courbe elliptique avec une différentielle de type $(a_{0}+2;-a_{0}-2)$. Le lemme~\ref{lem:lissdeuxcomp} permet de conclure.
\end{proof}

Pour terminer, il est intéressant de noter que l'existence de différentielles holomorphes pour toute partition~$\mu$ de $2g-2$ prouvée par \cite{masm} se déduit des deux propositions précédentes et du fait que la strate $\omoduli[2](2)$ est non vide. En effet, pour montrer que les strates de la forme $\omoduli(2g-2)$ ne sont pas vides, on coud successivement $g-2$ anses à une différentielle de cette strate. Les strates avec $n\geq2$ zéros s’obtiennent en éclatant le zéro de ces strates. Ces deux chirurgies sont introduites dans la Section 4.2 de \cite{kozo}.

\subsection{Différentielles à cœur dégénéré}
\label{sec:coeur}

Les différentielles à cœur dégénéré constitue une famille particulièrement simple d'exemples de différentielles abéliennes. En particulier, beaucoup de problèmes géométriques se simplifient sur ces différentielles en des problèmes combinatoires.

Rappelons que le {\em cœur} d'une différentielle est l'enveloppe convexe des singularités coniques pour la métrique induite par la différentielle. On dit que le cœur est \textit{dégénéré} s'il est d'intérieur vide, c'est-à-dire s'il est réduit à l'union d'un nombre fini de liens selles.

Le complémentaire du cœur d'une surface de translations $S$ possède autant de composantes connexes que de pôles. On appelle \textit{domaine polaire} la composante à laquelle un pôle appartient. Le bord d'un domaine polaire est toujours formé par un nombre fini de liens selles (voir le lemme~2.1 de \cite{tahar}). De plus, si le cœur de $S$ est dégénéré, alors il y a exactement $2g+n+p-2$ connexions de selles dans $S$, où $n$ est le nombre de zéros et~$p$ le nombre de pôles de~$S$.

L'intérêt de cette notion pour notre problème est donné par la proposition suivante.
\begin{prop}\label{prop:coeurdege}
Dans une strate donnée, le lieu des différentielles abéliennes dont tous les résidus sont nuls est soit vide soit contient une différentielle à cœur dégénéré.
\end{prop}

\begin{proof}
Le lieu d'une strate où les résidus sont tous nuls est ${\rm GL}^{+}(2,\mathbb{R})$-invariant. Or, d'après le lemme 2.2 de \cite{tahar}, chaque ${\rm GL}^{+}(2,\mathbb{R})$-orbite contient une surface plate dont le cœur est dégénéré.
\end{proof}

% Étant donnée une différentielle dont le cœur est dégénéré. On peut supposer, quitte à faire une rotation, que toutes les connexions de selles sont horizontales. De plus, il y a exactement $2g+n+p-2$ connexions de selles dans la surface $S$, où $n$ est le nombre de zéros et~$p$ le nombre de pôles dans la strate.

% Coupant le long de ces liens selles, nous obtenons~$p$ domaines polaires. Le {\em graphe d'incidence d'une surface à cœur dégénéré} est le graphe dont les sommets sont les domaines polaires et deux sommets sont reliés par autant d'arêtes  qu'il y a de liens selles entre les deux domaines polaires.
% De plus, le {\em graphe d'incidence simplifié} est  obtenu en enlevant tous les sommets de valence $2$ au graphe d'incidence. Les sommets du graphe d'incidence qui sont de valence supérieure ou égale à trois sont dits {\em spéciaux}.

\section{Les différentielles sur la sphère de Riemann}
\label{sec:g0}

Cette section est dédié au cas des différentielles sur une surface de Riemann de genre~$0$. Nous y donnons la preuve du Théorème~\ref{thm:geq0keq1}. Dans la section~\ref{sec:geng0}, nous traitons le cas général. Puis nous considérons dans la section~\ref{sec:casreszero} le cas où tous les résidus sont nuls. Enfin, nous traitons le cas des résidus colinéaires dans les strates dont tous les pôles sont simples dans la section~\ref{sec:caslier}.

\subsection{Le cas général}
\label{sec:geng0}
Dans cette section nous montrons l'existence de différentielles sur la sphère de Riemann possédant les invariants souhaités dans la majorité des cas. Plus précisément, nous  traitons les cas des résidus non nuls si il existe un pôle d'ordre $b_{i}\geq2$ et des résidus non tous colinéaires si tous les pôles sont simples.
\smallskip
\par
L'objet géométrique clef de la démonstration est le suivant. Soit $v:=(v_1,\ldots, v_{t})$ un uplet avec $v_{i}\neq 0$ et $\sum_{i=1}^{t}v_{i}=0$. Nous supposerons, quitte à permuter les indices, que l'argument des vecteurs $v_{1},\ldots,v_{t}$ est décroissant dans $\left]-\pi,\pi\right]$.
Le {\em polygone résiduel } $\mathfrak{P}(v)$ est le polygone obtenu en concaténant les vecteurs $v_{1}, \ldots,v_{t}$ dans cet ordre. Remarquons que~$\mathfrak{P}(v)$ est un polygone convexe, éventuellement dégénéré, c'est à dire d'intérieur vide.

Nous commençons par traiter le cas des résidus non tous colinéaires. Rappelons que deux éléments $v_{1}$ et $v_{2}$ de $\CC$ sont {\em colinéaires} si au moins un des $v_{i}$ est zéro ou s'il existe $\alpha\in \RR$ tel que $v_{1} = \alpha v_{2}$.
\begin{lem}\label{lem:gzerogen1}
 Soit $\mu:=(a;-b_{1},\dots,-b_{p};(-1^{s}))$ une partition  de $-2$. Si $r=(r_{1}, \dots , r_{p+s})\in\espres[0](\mu)$ est  un élément de l'espace résiduel de $\omoduli[0](\mu)$ qui possède deux éléments non colinéaires,  alors $r$ est dans l’image de l'application résiduelle  $\appres[0](\mu)$.
\end{lem}

\begin{proof}
 Notons par $r^{\ast}$ l'ensemble des résidus non nuls.  
 Comme les $r_{i}$ ne sont pas colinéaires, le polygone résiduel  $\mathfrak{P}(r^{\ast})$ est non dégénéré. Pour tous les pôles~$P_{i}$ d'ordre~$-b_{i}$ ayant un résidu non nul~$r_{i}$, on prend une partie polaire d'ordre~$b_{i}$ associée à $(r_{i};\emptyset)$ (et de type arbitraire). Pour tous les pôles $P_{j}$ d'ordre $-b_{j}$ ayant un résidu nul, on prend une partie triviale d'ordre $b_{j}$ associée à $(r_{i_{j}};r_{i_{j}})$, où $r_{i_{j}}\neq0$ est le résidu (non nul) au pôle $P_{i_{j}}$. Les collages de ces parties polaires avec le polygone résiduel se font de la façon suivante.

\'Etant donnée une partie polaire non triviale associée à un pôle $P_{i}$. S'il existe un pôle~$P_{j}$ qui à une partie polaire associée à $(r_{i};r_{i})$, alors nous collons le segment~$r_{i}$ de $P_{i}$ au segment~$r_{i}$ du domaine basique négatif de $P_{j}$. Nous continuons ces collages jusqu'à ce qu'il n'y ait plus de pôles sans résidu avec une partie polaire associée à $(r_{i};r_{i})$. Puis nous collons le dernier segment $r_{i}$ au segment $r_{i}$ du polygone résiduel. Nous faisons de même pour tous les pôles de résidus non nuls. Cette construction est illustrée par la figure~\ref{fig:casgeneralgenrezero}.

 \begin{figure}[htb]
 \center
\begin{tikzpicture}[scale=1.2]

%triangle
\begin{scope}[yshift=-.2cm]
\coordinate (a) at (0,0);
\coordinate (b) at (-1,0);
\coordinate (c) at (-1,-1);

\filldraw[fill=black!10] (a) -- (b)coordinate[pos=.5](d) -- (c)coordinate[pos=.5](e) -- (a) coordinate[pos=.5](f);

\fill (a)  circle (2pt);
\fill (b) circle (2pt);
\fill (c) circle (2pt);

\node[above] at (d) {$1$};
\node[left] at (e) {$2$};
\node[below right] at (f) {$3$};
\end{scope}

%pole 1
\begin{scope}[yshift=-.5cm]
\coordinate (a) at (-1,1);
\coordinate (b) at (0,1);

    \fill[fill=black!10] (a)  -- (b)coordinate[pos=.5](f) -- ++(0,1) --++(-1,0) -- cycle;
    \fill (a)  circle (2pt);
\fill[] (b) circle (2pt);
 \draw  (a) -- (b);
 \draw (a) -- ++(0,.9) coordinate (d)coordinate[pos=.5](h);
 \draw (b) -- ++(0,.9) coordinate (e)coordinate[pos=.5](i);
 \draw[dotted] (d) -- ++(0,.2);
 \draw[dotted] (e) -- ++(0,.2);
\node[above] at (f) {$1$};
\node[left] at (h) {$4$};
\node[right] at (i) {$4$};
\end{scope}

%pole2,1
\fill[fill=black!10] (-2.8,-.5) coordinate (Q) circle (1.1cm);
\coordinate (a) at (-2.8,0);
\coordinate (b) at (-2.8,-1);
\fill (a)  circle (2pt);
\fill (b) circle (2pt);
\draw (a) -- (b)coordinate[pos=.5](d);

\node[left] at (d) {$2$};
\node[right] at (d) {$5$};

%pole 2,2
\fill[fill=black!10] (-5.3,-.5) coordinate (Q) circle (1.1cm);
\coordinate (a) at (-5.4,0);
\coordinate (b) at (-5.4,-1);
\fill (a)  circle (2pt);
\fill (b) circle (2pt);
    \fill[white] (a) --  (b)coordinate[pos=.5](d) -- ++(1.3,0) --++(0,1) -- cycle;
 \draw  (a) -- (b)coordinate[pos=.5](g);
 \draw (a) -- ++(1,0) coordinate (e)coordinate[pos=.5] (h);
 \draw (b) -- ++(1,0) coordinate (f)coordinate[pos=.5] (i);
 \draw[dotted] (f) -- ++(.2,0);
 \draw[dotted] (e) -- ++(.2,0);
 
 \node[right] at (g) {$5$};
\node[above] at (h) {$6$};
\node[below] at (i) {$6$};

 %pole 3
 \begin{scope}[yshift=.1cm]
\fill[fill=black!10] (2,-.5) coordinate (Q) circle (1.3cm);
\coordinate (a) at (2.4,.1);
\coordinate (b) at (1.4,-.9);
\fill (a)  circle (2pt);
\fill (b) circle (2pt);
    \fill[white] (a) --  (b)coordinate[pos=.5](d) -- ++(0,1.8) --++(1,0) -- cycle;
 \draw  (a) -- (b)coordinate[pos=.5](g);
 \draw (a) -- ++(0,.5) coordinate (e)coordinate[pos=.6] (h);
 \draw (b) -- ++(0,1.4) coordinate (f)coordinate[pos=.6] (i);
 \draw[dotted] (f) -- ++(0,.2);
 \draw[dotted] (e) -- ++(0,.1);
 
 \node[below right] at (g) {$3$};
\node[left] at (h) {$9$};
\node[right] at (i) {$9$};
 
 \draw (a) -- ++(.7,0) coordinate[pos=.5] (j);
\node[above] at (j) {$8$};
\node[below] at (j) {$7$};
\end{scope}
%deuxieme figure

\fill[fill=black!10] (4.6,-.5) coordinate (Q) circle (1.1cm);

\draw[] (Q) -- ++(1.1,0) coordinate[pos=.5](a);

\node[above] at (a) {$7$};
\node[below] at (a) {$8$};
\fill[] (Q) circle (2pt);

\end{tikzpicture}
\caption{Une différentielle dans la strate $\omoduli[0](4;-2,-2,-3;-1)$ dont les résidus sont $(0,i,-1-i,1)$.} \label{fig:casgeneralgenrezero}
\end{figure}
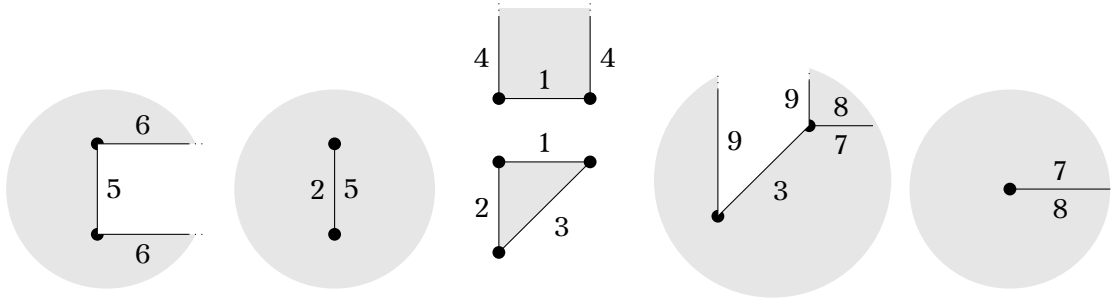

La différentielle associée à cette surface plate possède clairement les invariants désirés aux pôles. Vérifions maintenant qu'elle est de genre zéro et possède un unique zéro. Remarquons que si l'on coupe cette surface le long d'un lien selle, on obtient deux surfaces connexes disjointes. En effet, les liens selles correspondent soit au bord des parties polaires, soit aux diagonales du polygone résiduel. Une telle propriété implique qu'il existe un unique zéro, car s'il y en avait deux, un lien selle entre les deux ne déconnecterait pas la surface. De manière similaire, on en déduit que le genre de la surface est nul. 
\end{proof}

Nous traitons maintenant le cas des résidus colinéaires non tous nuls dans le cas où il existe au moins un pôle d'ordre supérieur ou égal à $2$.
\begin{lem}\label{lem:gzerogen2}
 Soit $\mu:=(a;-b_{1},\dots,-b_{p};(-1^{s}))$ une partition  de $-2$ telle que $p\neq0$. Un  élément $r=(r_{1}, \dots , r_{p+s})\in\espres[0](\mu)$
  non nul de l’espace résiduel est dans l’image de l'application résiduelle~$\appres[0](\mu)$.
\end{lem}

\begin{proof}
Par le lemme~\ref{lem:gzerogen1}, il suffit de considérer le cas où tous les $r_{i}$ sont colinéaires. Dans ce cas le polygone résiduel est dégénéré, mais comme les $r_{i}$ ne sont pas tous nuls ce polygone n'est pas réduit à un point. Sans perte de généralité, nous supposerons que les résidus sont réels. De plus, nous ordonnons les indices de telle sorte que l'ensemble des résidus non nuls $r^{\ast}=\left\{ r_{j_{1}},\dots,r_{j_{t}}\right\}$ satisfait $r_{j_{i}}<0$ pour $i\leq u$ et $r_{j_{i}}$ pour $u<i\leq t$. Nous notons $J=\left\{ j_{1},\dots,j_{t}\right\}$.

Prenons le pôle $P_{1}$ d'ordre $-b_{1}$. Si $r_{1}=0$, nous associons à ce pôle une partie polaire d'ordre $b_{i}$ associée aux vecteurs $(-r_{j_{1}},\dots,-r_{j_{u}};r_{j_{u+1}},\dots,r_{j_{t}})$. Si $r_{1}\neq 0$, on suppose que $r_{1}=r_{j_{1}}$ et on associe à ce pôle une partie polaire d'ordre $b_{1}$ associée aux vecteurs $(-r_{j_{2}},\dots,-r_{j_{u}};r_{j_{u+1}},\dots,r_{j_{t}})$.
Nous associons aux pôles $P_{j_{i}}$, pour $1\leq i\leq u$ si $1\notin J$ et $2\leq i\leq u$ si $1\in J$, une partie polaire d'ordre $b_{j_{i}}$ associée à $(\emptyset;-r_{j_{i}})$. Pour les $u< i\leq t$ nous prenons une partie polaire d'ordre $b_{i}$ associée à $(r_{i};\emptyset)$. Pour tous les pôles $P_{k}$ d'ordre $-b_{k}$ ayant un résidu nul, on prend une partie triviale d'ordre $b_{k}$ associée à $(\pm r_{i_{k}};\pm r_{i_{k}})$ avec $i_{k}\in J \setminus 1$. On colle les parties polaires triviales comme dans la preuve du lemme~\ref{lem:gzerogen1}. On obtient une union de surfaces plates à bord.  Il reste à coller les bords restants aux segments de la partie polaire de $P_{1}$.

La preuve du fait que la différentielle ainsi construite vérifie les propriétés désirées est en tout point similaire à la preuve de ce fait dans la démonstration du lemme~\ref{lem:gzerogen1}. 
\end{proof}

\subsection{Les résidus de tous les pôles sont nuls}
\label{sec:casreszero}
Nous complétons la preuve des cas i) et iii) du théorème~\ref{thm:geq0keq1}. Au vu des lemmes~\ref{lem:gzerogen1} et~\ref{lem:gzerogen2}, il reste à montrer le résultat suivant.
\begin{lem}\label{lem:condgzero}
Dans les strates $\omoduli[0](a_{1},\dots,a_{n};-b_{1},\dots,-b_{p})$ vérifiant $b_{1},\dots,b_{p} \geq 2$, l'image de l'application résiduelle ne contient pas l'origine $(0,\dots,0)$ si et seulement s'il existe un indice $i$ tel que $a_{i}>\sum_{j=1}^{p}b_{j}-(p+1)$.
\end{lem}

\begin{proof}
Nous commençons par montrer que si l'origine est dans l'image de l'application résiduelle alors tous les zéros sont d'ordres inférieurs ou égaux à $\sum_{j=1}^p b_{j}-(p+1)$.
Soit~$\omega$ une différentielle sur la sphère de Riemann de la strate $\omoduli[0](a_{1},\dots,a_{n};-b_{1},\dots,-b_{p})$ dont les résidus aux pôles sont nuls et $S$ la surface plate associée. Nous pouvons supposer par la proposition~\ref{prop:coeurdege} que la surface $S$ possède un cœur dégénéré et des liens selles horizontaux. Coupons $S$ le long de ces liens selles et de toutes les demi-droites horizontales issues des singularités. On obtient une alors une union disjointe de demi-plan positifs et négatifs. De plus, remarquons que chaque pôle d'ordre $-b_{j}$ est associé à $b_{j}-1$ demi-plan positifs et  $b_{j}-1$ demi-plan négatifs.

Considérons une singularité conique, disons $z_{1}$ d'angle $2\pi(a_{1}+1)$. Comme les résidus aux pôles sont nuls, on obtient directement du théorème des résidus que tout chemin fermé  de $S$ possède une période nulle.
Cela implique que les points correspondant à $z_{1}$ peuvent apparaître au plus une fois par demi-plan. 
Comme l'angle à chaque sommet dans chaque demi plan est $\pi$ l'angle maximal de la singularité conique $z_{1}$ est $2\pi\sum_{j=1}^p (b_{j}-1)$. Cela est équivalent au fait que l'ordre $a_{1}$ de $z_{1}$ est inférieur ou égal à $\sum_{j=1}^p b_{j}-(p+1)$.
\smallskip
\par
Nous montrons maintenant que si tous les zéros d'une différentielle sont d'ordres inférieurs ou égaux à $\sum b_{j}-(p+1)$ alors l'origine est dans l'image de l'application résiduelle.
Considérons tout d'abord les strates  $\omoduli[0]\left(a_{1},a_{2};-b_{1},\dots,-b_{p}\right)$ ayant deux zéros avec $p-1\leq a_{1},a_{2}\leq \sum b_{j} -(p+1)$. Pour tous les pôles nous prenons une partie polaire triviale $S_{i}$ d'ordre $b_{i}$ et de type~$\tau_{i}$ associée à $(1;1)$. Nous choisissons les~$\tau_{i}$ tels que $\sum_{i}\tau_{i}=a_{1}+1$. Ce choix est possible car pour chaque $i$ l'inégalité $1\leq \tau_{i} \leq b_{i}-1$ implique en sommant sur les pôles que $p\leq\sum_{i}\tau_{i}\leq \sum b_{j}-p$. 

Ensuite, nous collons les bords des parties polaires de manière cyclique. Plus précisément, nous collons le segment supérieur de $S_{i}$ au segment inférieur de $S_{i+1}$ modulo $p$. Une telle construction est représentée sur la figure~\ref{fig:ordremax}. La surface plate ainsi obtenue possède deux singularités coniques et est de genre nul. De plus, l'angle de la singularité conique à gauche des liens selles est d'angle $2\pi\sum \tau_{i}$. La différentielle ainsi construite appartient à la strate  $\omoduli[0]\left(a_{1},a_{2};-b_{1},\dots,-b_{p}\right)$ et les résidus aux pôles sont nuls.

 \begin{figure}[htb]
\center
\begin{tikzpicture}[scale=1.2]
%Figure haut gauche

\begin{scope}[xshift=-6cm]
\fill[fill=black!10] (0,0) coordinate (Q) circle (1.1cm);

\draw[] (0,0) coordinate (Q) -- (1.1,0) coordinate[pos=.5](a);

\node[above] at (a) {$1$};
\node[below] at (a) {$3$};

\draw[] (Q) -- (-.5,0) coordinate (P) coordinate[pos=.5](c);

\fill (Q)  circle (2pt);
%\fill (Q1)  arc [start angle=180,end angle=0,radius=2pt];
%\fill (Q2) arc [start angle=180,end angle=360,radius=2pt];

\fill[color=white!50!] (P) circle (2pt);
\draw[] (P) circle (2pt);
\node[above] at (c) {$a$};
\node[below] at (c) {$b$};
\end{scope}

%deuxieme dessin
\begin{scope}[xshift=-3.7cm]
\fill[fill=black!10] (0,0) coordinate (Q) circle (1.1cm);

\draw[] (0,0) -- (1.1,0) coordinate[pos=.5](a);

\node[above] at (a) {$2$};
\node[below] at (a) {$1$};
\fill (Q)  circle (2pt);
\end{scope}
%troisieme dessin
\begin{scope}[xshift=-1.4cm]
\fill[fill=black!10] (0,0) coordinate (Q) circle (1.1cm);

\draw[] (0,0) -- (1.1,0) coordinate[pos=.5](a);

\node[above] at (a) {$3$};
\node[below] at (a) {$2$};
\fill (Q)  circle (2pt);
\end{scope}

%dessin en dessous
\begin{scope}[xshift=2cm]
\fill[fill=black!10] (0,0)  circle (1.1cm);

\draw[] (0,0) coordinate (Q) -- (1.1,0) coordinate[pos=.5](a);

\node[above] at (a) {$4$};
\node[below] at (a) {$5$};

\draw[] (0,0) -- (-.5,0) coordinate (P) coordinate[pos=.5](c);

\fill (Q)  circle (2pt);

\fill[color=white!50!] (P) circle (2pt);
\draw[] (P) circle (2pt);
\node[above] at (c) {$b$};
\node[below] at (c) {$a$};
\end{scope}

%Deuxieme dessin en dessous
\begin{scope}[xshift=4.3cm]
\fill[fill=black!10] (0,0) coordinate (Q) circle (1.1cm);
    
\draw[] (0,0) -- (1.1,0) coordinate[pos=.5](a);

\node[above] at (a) {$5$};
\node[below] at (a) {$4$};
\fill (Q)  circle (2pt);
\end{scope}
\end{tikzpicture}
\caption{Une différentielle dans $\omoduli[0](4,1;-3,-4)$ dont les résidus sont nuls.} \label{fig:ordremax}
\end{figure}
\smallskip
\par
Nous traitons maintenant les strates $\omoduli[0](a_{1},a_{2},a_{3};-b_{1},\dots,-b_{p})$, avec $a_{1}, a_{2}\leq a_{3}$. Il y a deux cas à considérer suivant que  $a_{1},a_{2}>\sum b_{i}-p-1$ ou non.
\par
Si $a_{1}+a_{2}\leq \sum b_{i}-p-1$ (et donc $p-1\leq a_{3} \leq \sum b_{i}-p-1$), alors il existe une différentielle à résidus nuls dans la strate $\omoduli[0](a_{1}+a_{2},a_{3};-b_{1},\dots,-b_{p})$. Donc en éclatant le zéro d'ordre $a_{1}+a_{2}$ en deux zéros d'ordres $a_{1}$ et $a_{2}$ par la proposition~\ref{prop:eclatZero}, nous obtenons la différentielle souhaitée.
\par
Supposons maintenant que $a_{1}+a_{2} > \sum b_{i}-p-1$ (ou de manière équivalente $a_{3}<p-1$). Dans le reste de la preuve nous noterons $b:=\sum b_{i}$. Remarquons que  $a_{1},a_{2}\leq a_{3}$ et $a_{1}+a_{2}+a_{3}=b-2$ impliquent que  $ 3a_{3} \geq b-2$. Nous obtenons donc que $3p>b+1$, ce qui implique que l'un des pôles est d'ordre~$-2$.
\par
Nous donnons maintenant la description d'une différentielle ayant les propriétés attendues. Cette construction est illustrée dans la figure~\ref{fig:troiszerossansres} dans le cas de  $\omoduli[0](3^3,-2^4,-3)$. Dans un premier temps, nous décrivons le procédé, puis nous ajusterons les constantes pour obtenir les singularités coniques souhaitées. 

Pour le pôle d'ordre $-2$, prenons une partie polaire triviale d'ordre $2$ associée à $(v_{1},v_{2};v_{3})$ avec $v_{3}=v_{1}+v_{2}$.  Pour chaque pôle $P_{i}$ nous prenons une partie polaire triviale d'ordre~$b_{i}$ de type $\tau_{i}$ associée  à  $(v_{j_{i}};v_{j_{i}})$ pour un $j_{i}\in\lbrace 1,2,3 \rbrace$. Puis nous collons le segment $v_{j_{i}}$ du triangle au segment $v_{j_{i}}$ correspondant au domaine basique négatif de cette partie polaire.  Cette opération ajoute une contribution angulaire de $2\pi\tau_{i}$ et $2\pi(b_{i}-\tau_{i})$ aux singularités coniques correspondant aux sommets du segment $v_{j_{i}}$. Nous faisons de même pour tous les pôles jusqu'à obtenir une surface plate $S_{1}$ dont le bord est composé des trois segments $v_{j}$. Nous prenons maintenant le triangle $v_{1}v_{2}v_{3}$ et collons par translation ses trois arêtes au bord de $S_{1}$.
La surface plate ainsi obtenue est de genre zéro, possède trois zéros distincts et n'a pas de résidus aux pôles. 
\par
Il reste à ajuster le choix  des segments $v_{j_{i}}$ et des types $\tau_{i}$ pour chaque pôle afin d'obtenir les  angles souhaités. La remarque clé est que chaque pôle $P_{i}$ contribue à exactement deux singularités coniques et que la contribution à chacune d'elles est d'angle compris entre $2\pi$ et $2\pi(b_{i}-1)$. Réciproquement, si cette condition est satisfaite, alors la construction précédente permet de construire la différentielle souhaitée.
\par
La situation peut donc être modélisée par un graphe biparti $\Gamma$ dont trois sommets~$A_{i}$ représentent les trois sommets du triangle $v_{1}v_{2}v_{3}$ et $p-1$ autres sommets $B_{i}$ représentent les pôles $P_{i}$ distincts du pôle avec la partie polaire associée au triangle. Il y a une arête entre les sommets $B_{i}$ et~$A_{j}$ pour chaque multiple de $2\pi$ de la contribution de $P_{i}$ à la singularité conique~$A_{j}$. Un exemple est schématisé dans la  figure~\ref{fig:troiszerossansres}. 
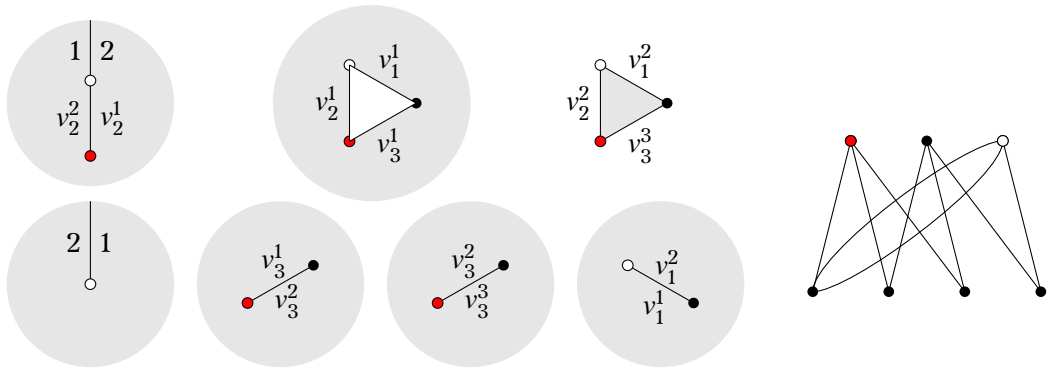
\begin{figure}[htb]
\center
\begin{tikzpicture}
%Figure haut gauche

\begin{scope}[xshift=-1.3cm]
\fill[fill=black!10] (0,0)  circle (1.3cm);
   
 \coordinate (A) at (1/1.71,0);
  \coordinate (B) at (120:1/1.71);
    \coordinate (C)  at (240:1/1.71);
    
    \fill (A)  circle (2pt);
\filldraw[fill=white] (B) circle (2pt);
\filldraw[fill=red] (C)  circle (2pt);

   \fill[color=white]     (A) -- (B) -- (C) --cycle;
 \draw[]     (A) -- (B) coordinate[pos=.4](a);
 \draw (B) -- (C) coordinate[pos=.5](b);
 \draw (C) -- (A)  coordinate[pos=.6](c);

\node[above] at (a) {$v_{1}^{1}$};
\node[left] at (b) {$v_{2}^{1}$};
\node[below ] at (c) {$v_{3}^{1}$};

\end{scope}

%deuxieme dessin
\begin{scope}[xshift=2cm]
 \coordinate (A) at (1/1.71,0);
  \coordinate (B) at (120:1/1.71);
    \coordinate (C)  at (240:1/1.71);
    
    \fill[fill=black!10]     (A) -- (B) -- (C) --cycle;

 \draw[]     (A) -- (B) coordinate[pos=.4](a);
 \draw (B) -- (C) coordinate[pos=.5](b);
 \draw (C) -- (A)  coordinate[pos=.6](c);

\node[above] at (a) {$v_{1}^{2}$};
\node[left] at (b) {$v_{2}^{2}$};
\node[below ] at (c) {$v_{3}^{3}$};

 %   \begin{scope}[]
  % \clip (A) -- (B) -- (C) -- cycle;

    \fill (A)  circle (2pt);

\filldraw[fill=white] (B) circle (2pt);
\filldraw[fill=red] (C)  circle (2pt);

 %\end{scope}
\end{scope}

%%%%%%%%%%%%%%%%%%%%%%%%%%%%%%%%%%%%%%%%%%%%%%%%%%%%%%%%%%%%%%%%%%%%%%%%%%%%%%%%%%%%%%%%%%%%%%%%%%%%%
% Ici les dessins du bas!!!!!
 \begin{scope}[xshift=-5cm]
\fill[fill=black!10] (0,0)  circle (1.1cm);
 
\draw[](0,-.7)coordinate(P) --(0,.3) coordinate[pos=.5](b)coordinate (Q) -- (0,1.1) coordinate[pos=.5](a);

\node[left] at (a) {$1$};
\node[right] at (a) {$2$};

\filldraw[fill=red] (P)  circle (2pt);
%\begin{scope}[yshift=-4pt]
%\fill[white] (Q1)  arc [start angle=-90,end angle=90,radius=2pt];
%\draw (Q1)  arc [start angle=-90,end angle=90,radius=2pt];

%\fill[white] (Q2)  arc [start angle=270,end angle=90,radius=2pt];
%\draw (Q2)  arc [start angle=270,end angle=90,radius=2pt];
%\end{scope}
\filldraw[fill=white] (Q) circle (2pt);
\node[left] at (b) {$v_{2}^{2}$};
\node[right] at (b) {$v_{2}^{1}$};
\end{scope}
\begin{scope}[yshift=-2.4cm]

%deuxieme dessin
\begin{scope}[xshift=-5cm]
\fill[fill=black!10] (0,0)  circle (1.1cm);

\draw[] (0,0) coordinate (Q1) -- (0,1.1) coordinate[pos=.5](a);

\node[left] at (a) {$2$};
\node[right] at (a) {$1$};

%\fill (Q)  circle (2pt);
%\begin{scope}[yshift=-2pt]
%\fill (Q1)  arc [start angle=270,end angle=450,radius=2pt];
%\end{scope}

\filldraw[fill=white] (Q1) circle (2pt);

\end{scope}

%troisieme figure

\begin{scope}[xshift=-2.5cm]
\fill[fill=black!10] (0,0) coordinate (Q) circle (1.1cm);
    
    \begin{scope}[rotate=120]
\draw[]     (0,-1/2) -- (0,1/2)coordinate[pos=.6](a);
\node[above] at (a) {$v_{3}^{1}$};
\node[xshift=5pt,yshift=-5pt] at (a) {$v_{3}^{2}$};

    \fill (0,-1/2)  circle (2pt);
    \filldraw[fill=red] (0,1/2)  circle (2pt);

    \end{scope}

\end{scope}

%quatrieme dessin
\begin{scope}[xshift=0cm]
\fill[fill=black!10] (0,0) coordinate (Q) circle (1.1cm);
    
    \begin{scope}[rotate=120]
\draw[]     (0,-1/2) -- (0,1/2)coordinate[pos=.6](a);
\node[above] at (a) {$v_{3}^{2}$};
\node[xshift=5pt,yshift=-5pt] at (a) {$v_{3}^{3}$};

    \fill (0,-1/2)  circle (2pt);
    \filldraw[fill=red] (0,1/2)  circle (2pt);

    \end{scope}

\end{scope}

%cinquieme
\begin{scope}[xshift=2.5cm]
\fill[fill=black!10] (0,0) coordinate (Q) circle (1.1cm);
      
    \begin{scope}[rotate=240]
\draw[]     (0,-1/2) -- (0,1/2)coordinate[pos=.4](a);
\node[below] at (a) {$v_{1}^{1}$};
\node[xshift=5pt,yshift=5pt] at (a) {$v_{1}^{2}$};

\fill (0,1/2)  circle (2pt);
\filldraw[fill=white] (0,-1/2) circle (2pt);
    \end{scope}
\end{scope}
\end{scope}

\begin{scope}[xshift=6cm,yshift=-.5cm]
\filldraw[fill=red] (-1,0)coordinate (A1) circle (2pt);
\fill (0,0)coordinate (A2)  circle (2pt);
\filldraw[fill=white] (1,0)coordinate (A3) circle (2pt);

\fill (-1.5,-2)coordinate (B1)  circle (2pt);
\fill (-.5,-2)coordinate (B2)  circle (2pt);
\fill (.5,-2)coordinate (B3)  circle (2pt);
\fill (1.5,-2)coordinate (B4)  circle (2pt);

\draw (A1) -- (B1);
\draw (A1) -- (B2);
\draw (A1) -- (B3);
\draw (A2) -- (B2);
\draw (A2) -- (B3);
\draw (A2) -- (B4);
\draw (A3) -- (B4);
\draw (A3) .. controls ++(-90:.5) and ++(0:.5) .. (B1);
\draw (A3) .. controls ++(180:.5) and ++(90:.5) .. (B1);

\filldraw[fill=red] (A1) circle (2pt);
\filldraw[fill=white] (A3) circle (2pt);
    \end{scope}

\end{tikzpicture}
\caption{Une différentielle dans $\omoduli[0]((3^3);(-2^4),-3)$ avec résidus nuls et son graphe $\Gamma$.} \label{fig:troiszerossansres}
\end{figure}

Il suffit de montrer qu'il est possible de distribuer les arêtes de telle façon que les sommets~$B_{i}$ soient de valence $b_{i}$, connectés à précisément deux sommets $A_{j}$, et que la valence de $A_{j}$ soit $a_{j}$.
Une telle distribution peut être obtenue de la façon suivante. Rappelons que le plus grand zéro (que nous supposerons être $z_{3}$) est d'ordre $a_{3}\leq p-1$. Nous partons du graphe où tous les sommets $B_{i}$ sont connectés à $A_{1}$ par exactement une arête et toutes les autres arêtes connectent $A_{2}$. Prenons un sommet $B_{i}$ quelconque. Il y a $b_{i}-1$ arêtes entre $B_{i}$ et $A_{2}$. Si la valence de $A_{2}$ moins $b_{i}-1$ est supérieure ou égale à~$a_{2}$, alors nous remplaçons les $b_{i}-1$ arêtes entre $B_{i}$ et  $A_{2}$ par $b_{i}-1$ arêtes entre $B_{i}$ et~$A_{3}$. Nous recommençons alors cette opération jusqu'à l'indice $i_{0}$ tel que la valence de $A_{2}$ moins $b_{i_{0}}-1$ soit strictement inférieure à~$a_{2}$. Dans ce cas, nous remplaçons des arêtes entre $B_{i_{0}}$ et $A_{2}$ par des arêtes entre $B_{i_{0}}$ et~$A_{3}$, de telle façon que la valence de $A_{2}$  soit égale à~$a_{2}$. Le sommet $B_{i_{0}}$ est alors connecté aux trois sommets $A_{i}$. Donc nous enlevons l'arête entre $A_{1}$ et $B_{i_{0}}$ pour la mettre entre $A_{3}$ et $B_{i_{0}}$. Comme $a_{1}$ est strictement plus petit que $p-1$, cette opération est toujours possible.  Pour terminer, nous remplaçons autant d'arêtes que nécessaire entre $A_{1}$ et les $B_{j}$ pour $j > i_{0}$ pour les connecter à $A_{3}$ afin que la valence de $A_{1}$ soit~$a_{1}$. Notons qu'il existe toujours assez d'arêtes entre $A_{1}$ et les $B_{j}$ pour obtenir la valence $a_{3}$ à $A_{3}$ car il y a un total de $p-1$ sommets~$B_{i}$ et $a_{3}\leq p-1$.

Pour conclure le cas des différentielles de genre zéro avec des résidus nuls, nous considérons les strates ayant $n\geq4$ zéros.  Soient~$a_{1}$ et~$a_{2}$ les zéros de plus petits ordres. En notant  $b=\sum_{i} b_{i}$ on~a  
\begin{equation*}
 a_{1}+a_{2} \leq \frac{2}{n}\left(b-2\right)\leq \sum_{i}(b_{i}-1) -\frac{4}{n} < b -p,
\end{equation*}
 où la deuxième égalité s'obtient en remarquant que  $\tfrac{2}{n} \leq \tfrac {1}{2}$ et $b_{i}\geq2$ impliquent que $\tfrac{2b_{i}}{n}\leq b_{i}-1$. On en déduit que $ a_{1}+a_{2} \leq b -p - 1$ et donc que le cas $n\geq 4$ s'obtient en éclatant un zéro d'une différentielle ayant $n-1$ zéros d'ordres $a_{1}+a_{2}$, $a_{i}$ pour $i\geq 3$ et $p$ pôles d'ordres~$-b_{j}$.
\end{proof}

\subsection{Les pôles sont tous simples}
\label{sec:caslier}
Dans cette section, nous montrons le point ii) du théorème~\ref{thm:geq0keq1} qui traite du cas des différentielles en genre zéro n'ayant que des pôles simples. 

\smallskip
\par
Nous commençons par le cas des strates $\omoduli[0](s-2;(-1^{s}))$. Nous définissons deux types de graphes que nous utilisons de manière essentielle dans la preuve.

\begin{defn}
  Un {\em graphe de connexion} est un arbre biparti connexe $\Gamma$ possédant $A$ arêtes, dont les sommets sont partitionnés en $\Gamma_{-}\cup\Gamma_{+}$ et auxquels sont attribués des poids réels strictement positifs, tels que:
\begin{enumerate}[i)]
 \item la somme des poids des sommets de $\Gamma_{+}$ est égale à celle des poids des sommets de~$\Gamma_{-}$;
 \item considérant l'opération qui consiste à retirer une feuille et soustraire le poids de ce sommet à celui qui lui est relié, appliquant cette opération entre une et $A-1$ fois à $\Gamma$, on obtient alors des graphes dont les poids sont strictement positifs.
\end{enumerate}
\end{defn}

\begin{defn}
Soit $r_{1},\dots,r_{s}$ des nombres complexes $\RR$-colinéaires de somme nulle,  un {\em graphe associé aux~$r_{i}$} est un arbre biparti connexe vérifiant les propriétés suivantes.  \'Etant donné $\alpha\in\CC^{\ast}$ tel que $r_{i}':=\alpha r_{i}\in\RR^{\ast}$ pour tout $i\leq s$, les sommets de $\Gamma_{+}$ (resp. $\Gamma_{-}$) sont en bijection avec les $r_{i}'$ positifs (resp. négatifs) et le poids du sommet correspondant à $r_{i}$ est $|r_{i}'|$.
\end{defn}

L'image de l'application résiduelle des strates $\omoduli[0](s-2;(-1^{s}))$ est décrite par le résultat suivant en terme de graphes.
\begin{lem}\label{lem:gzeropolesimples}
 Soit $\omoduli[0](s-2;(-1^{s}))$ une strate de genre zéro avec $s$ pôles simples et un unique zéro d'ordre $s-2$. Les nombres complexes $(r_{1},\ldots,r_{s})$ sont dans l'image de l'application résiduelle si et seulement si 
 l'une des propriétés est satisfaite.
 \begin{enumerate}
  \item Les nombres $(r_{1},\ldots,r_{s})$ ne sont pas colinéaires.
  \item  Les nombres $(r_{1},\ldots,r_{s})$ sont colinéaires et il existe un graphe associé aux $r_{i}$ qui est un graphe de connexion.
 \end{enumerate}
\end{lem}

Avant de passer à la preuve de ce résultat, nous illustrons ces notions dans un exemple.
Nous illustrons maintenant les concepts introduits dans un exemple.
\begin{ex} \label{ex:graphedeconnexion}
 Nous illustrons la correspondance entre une différentielle de $\omoduli[0](5;(-1^7))$ et le graphe de connexion associé. La figure~\ref{fig:totliermaisdansimage} montre que l'image de l'application résiduelle de  cette strate contient les résidus $(3,(1^{3}),(-2^{3}))$.
 
\begin{figure}[htb]
\begin{tikzpicture}[scale=1.1]
%pole 1,1
\begin{scope}[xshift=-6cm]
\coordinate (a) at (-1,1);
\coordinate (b) at (2,1);
\coordinate (c) at (0,1);
\coordinate (d) at (1,1);

    \fill[fill=black!10] (a)  -- (c)coordinate[pos=.5](f)-- (d)coordinate[pos=.5](g)-- (b)coordinate[pos=.5](j) -- ++(0,1.2) --++(-3,0) -- cycle;
    \fill (a)  circle (2pt);
\fill[] (b) circle (2pt);
    \fill (c)  circle (2pt);
\fill[] (d) circle (2pt);
 \draw  (a) -- (b);
 \draw (a) -- ++(0,1.1) coordinate (d)coordinate[pos=.5](h);
 \draw (b) -- ++(0,1.1) coordinate (e)coordinate[pos=.5](i);
 \draw[dotted] (d) -- ++(0,.2);
 \draw[dotted] (e) -- ++(0,.2);
\node[below] at (f) {$1$};
\node[below] at (g) {$2$};
\node[below] at (j) {$3$};
\end{scope}

%pole 1,2
\begin{scope}[xshift=-2.5cm]
\coordinate (a) at (-1,1);
\coordinate (b) at (0,1);

    \fill[fill=black!10] (a)  -- (b)coordinate[pos=.5](f) -- ++(0,1.2) --++(-1,0) -- cycle;
    \fill (a)  circle (2pt);
\fill[] (b) circle (2pt);
 \draw  (a) -- (b);
 \draw (a) -- ++(0,1.1) coordinate (d)coordinate[pos=.5](h);
 \draw (b) -- ++(0,1.1) coordinate (e)coordinate[pos=.5](i);
 \draw[dotted] (d) -- ++(0,.2);
 \draw[dotted] (e) -- ++(0,.2);
\node[below] at (f) {$4$};
\end{scope}

%pole 1,3
\begin{scope}[xshift=-1cm]
\coordinate (a) at (-1,1);
\coordinate (b) at (0,1);

    \fill[fill=black!10] (a)  -- (b)coordinate[pos=.5](f) -- ++(0,1.2) --++(-1,0) -- cycle;
    \fill (a)  circle (2pt);
\fill[] (b) circle (2pt);
 \draw  (a) -- (b);
 \draw (a) -- ++(0,1.1) coordinate (d)coordinate[pos=.5](h);
 \draw (b) -- ++(0,1.1) coordinate (e)coordinate[pos=.5](i);
 \draw[dotted] (d) -- ++(0,.2);
 \draw[dotted] (e) -- ++(0,.2);
\node[below] at (f) {$5$};
\end{scope}

%pole 1,4
\begin{scope}[xshift=.5cm]
\coordinate (a) at (-1,1);
\coordinate (b) at (0,1);

    \fill[fill=black!10] (a)  -- (b)coordinate[pos=.5](f) -- ++(0,1.2) --++(-1,0) -- cycle;
    \fill (a)  circle (2pt);
\fill[] (b) circle (2pt);
 \draw  (a) -- (b);
 \draw (a) -- ++(0,1.1) coordinate (d)coordinate[pos=.5](h);
 \draw (b) -- ++(0,1.1) coordinate (e)coordinate[pos=.5](i);
 \draw[dotted] (d) -- ++(0,.2);
 \draw[dotted] (e) -- ++(0,.2);
\node[below] at (f) {$6$};
\end{scope}
%pole 2,1
\begin{scope}[xshift=-5.75cm,yshift=-1cm]
\coordinate (a) at (-1,1);
\coordinate (b) at (1,1);
\coordinate (c) at (0,1);

    \fill[fill=black!10] (a)  -- (b)coordinate[pos=.25](f)coordinate[pos=.75](g) -- ++(0,-1.2) --++(-2,0) -- cycle;
    \fill (a)  circle (2pt);
\fill[] (b) circle (2pt);
    \fill (c)  circle (2pt);
 \draw  (a) -- (b);
 \draw (a) -- ++(0,-1.1) coordinate (d)coordinate[pos=.5](h);
 \draw (b) -- ++(0,-1.1) coordinate (e)coordinate[pos=.5](i);
 \draw[dotted] (d) -- ++(0,-.2);
 \draw[dotted] (e) -- ++(0,-.2);
\node[above] at (f) {$1$};
\node[above] at (g) {$4$};
\end{scope}

%pole 2,2
\begin{scope}[xshift=-3.25cm,yshift=-1cm]
\coordinate (a) at (-1,1);
\coordinate (b) at (1,1);
\coordinate (c) at (0,1);

    \fill[fill=black!10] (a)  -- (b)coordinate[pos=.25](f)coordinate[pos=.75](g) -- ++(0,-1.2) --++(-2,0) -- cycle;
    \fill (a)  circle (2pt);
\fill[] (b) circle (2pt);
    \fill (c)  circle (2pt);
 \draw  (a) -- (b);
 \draw (a) -- ++(0,-1.1) coordinate (d)coordinate[pos=.5](h);
 \draw (b) -- ++(0,-1.1) coordinate (e)coordinate[pos=.5](i);
 \draw[dotted] (d) -- ++(0,-.2);
 \draw[dotted] (e) -- ++(0,-.2);
\node[above] at (f) {$2$};
\node[above] at (g) {$5$};
\end{scope}

%pole 2,3
\begin{scope}[xshift=-.75cm,yshift=-1cm]
\coordinate (a) at (-1,1);
\coordinate (b) at (1,1);
\coordinate (c) at (0,1);

    \fill[fill=black!10] (a)  -- (b)coordinate[pos=.25](f)coordinate[pos=.75](g) -- ++(0,-1.2) --++(-2,0) -- cycle;
    \fill (a)  circle (2pt);
\fill[] (b) circle (2pt);
    \fill (c)  circle (2pt);
 \draw  (a) -- (b);
 \draw (a) -- ++(0,-1.1) coordinate (d)coordinate[pos=.5](h);
 \draw (b) -- ++(0,-1.1) coordinate (e)coordinate[pos=.5](i);
 \draw[dotted] (d) -- ++(0,-.2);
 \draw[dotted] (e) -- ++(0,-.2);
\node[above] at (f) {$3$};
\node[above] at (g) {$6$};
\end{scope}

%Graphe
\begin{scope}[xshift=3cm,yshift=1.5cm]
\filldraw[fill=white] (-1,-2)coordinate (A1) circle (2pt);\node[below] at (A1) {$2$};
\filldraw[fill=white] (0,-2)coordinate (A2)  circle (2pt);\node[below] at (A2) {$2$};
\filldraw[fill=white] (1,-2)coordinate (A3) circle (2pt);\node[below] at (A3) {$2$};

\fill (-1.5,0)coordinate (B1)  circle (2pt);\node[above] at (B1) {$3$};
\fill (-.5,0)coordinate (B2)  circle (2pt);\node[above] at (B2) {$1$};
\fill (.5,0)coordinate (B3)  circle (2pt);\node[above] at (B3) {$1$};
\fill (1.5,0)coordinate (B4)  circle (2pt);\node[above] at (B4) {$1$};

\draw (A1) -- (B1);
\draw (A2) -- (B1);
\draw (A3) -- (B1);
\draw (A1) -- (B2);
\draw (A2) -- (B3);
\draw (A3) -- (B4);

\filldraw[fill=white] (-1,-2)coordinate (A1) circle (2pt);
\filldraw[fill=white] (0,-2)coordinate (A2)  circle (2pt);
\filldraw[fill=white] (1,-2)coordinate (A3) circle (2pt);
    \end{scope}

\end{tikzpicture}
\caption{Une différentielle dans $\omoduli[0](5;(-1^7))$ avec résidus $\left(3,(1^{3}),(-2^{3})\right)$ et son graphe de connexion.} \label{fig:totliermaisdansimage}
\end{figure}
\end{ex}

\begin{proof}[Démonstration du lemme~\ref{lem:gzeropolesimples}]
Soit $r:=(r_{1},\dots,r_{s})$ un élément de l'espace résiduel de la strate $\omoduli[0](s-2;(-1^{s}))$ avec $s\geq2$. Si tous les résidus ne sont pas colinéaires alors le polygone résiduel introduit dans la section~\ref{sec:geng0} est non dégénéré. On obtient la différentielle abélienne avec les invariants souhaités en collant des cylindres infinis à toutes les arêtes de ce polygone. 

Nous supposerons donc à partir de maintenant que les $r_{i}$ sont colinéaires. Supposons qu'il existe un graphe $\Gamma$ associé à $r$ qui soit un graphe de connexion. On construit une différentielle de la façon suivante. Pour chaque résidu $r_{i}$ on prend une partie polaire d'ordre $1$ associée à $r_{i}$. Considérons une feuille de $\Gamma$. On peut coller le segment au bord de la partie polaire correspondante au segment au bord de la partie polaire correspondant à l'autre sommet de l'arête. Puis on enlève la feuille du graphe et le poids de cette feuille à l'autre sommet. On recommence cette procédure pour une feuille du nouveau graphe.  Cette opération est faite de manière inductive jusqu'à ce que le graphe soit réduit à un sommet.

Nous justifions maintenant que cette opération est toujours possible. Supposons tout d'abord que le graphe possède strictement plus de deux sommets. Comme par le point~ii) le poids d'une feuille est toujours strictement plus petit que le poids du sommet auquel elle est reliée, le collage est toujours possible. Enfin le point~i) dit que la différence des poids est nulle, ce qui implique que la surface obtenue est sans bord. On vérifie maintenant que cette surface plate est de genre zéro et possède un unique zéro. La surface est formée de domaines polaires reliés par des liens selles tous horizontaux. La surface est de genre zéro car s'il existait un lacet fermé homotopiquement non trivial, cela impliquerait un chemin fermé sur le graphe de connexion. Or celui-ci est un arbre. L'arbre étant connexe, la surface l'est également. Il reste à montrer que les extrémités des liens selles constituent une unique singularité conique. En tant que graphe plongé dans la surface, le graphe des liens selles est dual de celui défini par le graphe de connexion avec des arcs reliant les pôles. Le graphe de connexion étant un arbre, il ne définit qu'une seule face. Par conséquent, le graphe des liens selles n'a qu'un seul sommet.
\smallskip
\par
Supposons maintenant qu'il existe une différentielle $\omega$ dans $\omoduli[0](s-2;(-1^{s}))$ dont les résidus sont $(r_{1},\dots,r_{s})$. Supposons que les résidus soient colinéaires, nous les supposerons réels sans perte de généralité. Nous construisons un graphe associé aux $r_{i}$ qui est un graphe de connexion. Les sommets de $\Gamma_{+}$ (resp. $\Gamma_{-}$) sont associés aux pôles de $\omega$ dont le résidu est positif (resp. négatif). Les poids sont les valeurs absolues des résidus. Enfin deux sommets sont connectés si et seulement si le bord de leurs domaines polaires respectifs contiennent un même lien selle.

Le fait que $\omega$ soit de genre zéro et ne possède qu'un zéro implique clairement que ce graphe biparti est un arbre. Le théorème des résidus implique directement le point i) qui dit que la différence des poids est nulle. Regardons maintenant l'effet de l'opération qui enlève une feuille au graphe de connexion (dans le cas où il possède au moins deux arêtes). Cette opération revient à couper la différentielle $\omega$ le long d'un lien selle $v$ dont l'un des côté est un unique domaine polaire. De plus, le domaine polaire adjacent à ce lien selle est remplacé par le domaine polaire associé aux vecteurs précédents privés de~$v$ (qui est un ensemble non vide car $\omega$ est non singulière). Les autres domaines polaires et les identifications restent les mêmes. 
%Le domaine polaire de l'autre côté est bordé d'autres liens selles. Dans le cas contraire, la surface serait singulière. Cela est vrai à toutes les étapes montrant le point (ii) de la définition des graphes de connexion. 
\`A chaque étape, cette opération produit une différentielle de genre zéro avec un unique zéro et des pôles simples dont tous les résidus sont non nuls. Cela implique clairement le point ii) des graphes de connexion.
\end{proof}

Nous montrons maintenant que les $s$-uplets qui ne sont pas dans l'image de  l'application résiduelle $\appres[0](s-2;(-1^{s}))$ sont commensurables entre eux.
\begin{lem}\label{lem:finietcommens}
Soient $r:=(x_{1},\dots,x_{s_{1}},-y_{1},\dots,-y_{s_{2}})$ avec $x_{i}$ et $y_{j}$ réels strictement positifs. Si le $s$-uplet $r$ n'appartiennent pas à l'image de $\appres[0](s-2;(-1^{s}))$, alors les $x_{i}$ et $y_{j}$ sont commensurables entre eux. 
\end{lem}
% En particulier, il n'y a, à homothétie près, qu'un nombre fini de $s$-uplets qui ne sont pas dans l'image de $\appres[0](s-2;(-1^{s}))$.

\begin{proof}
Nous procédons à une démonstration par récurrence.
Si on a $s_{1}=1$ ou $s_{2}=1$, toutes les configurations de résidus sont réalisables. Si $s_{1}=s_{2}=2$, les seuls uplets qui ne sont pas dans l'image de l'application résiduelle sont proportionnels à $(1,1,-1,-1)$.

\`A présent, on suppose que la proposition est démontrée pour tous les couples $(a,b)\neq (s_{1},s_{2})$ tels que $a \leq s_{1}$ et $b \leq s_{2}$. Nous considérons un uplet qui n'est pas dans l'image avec $s_{1}$ nombres positifs et $s_{2}$ négatifs. Si tous les $x_{i}$ et $y_{j}$ sont égaux, alors on a $s_{1}=s_{2}$ et le uplet est proportionnel à $(1,\dots,1,-1,\dots,-1)$.
% Les résidus sont commensurables et la somme des résidus d'une même série respecte la borne de l'équation~\eqref{eq:polessimples}.

On peut donc supposer qu'il existe deux résidus, disons $x_{s_{1}}$ et $-y_{s_{2}}$, tels que $y_{s_{2}}<x_{s_{1}}$. Le $(s-1)$-uplet obtenu en retirant le résidu $-y_{s_{2}}$ et en remplaçant $x_{s_{1}}$ par $x_{s_{1}}-y_{s_{2}}$ n'est pas réalisable. En effet, si ce nouveau uplet était réalisable, il existerait un graphe de connexion qui lui serait associé. Il suffirait d'ajouter à ce graphe une branche avec comme poids $y_{s_{2}}$ au sommet de poids  $x_{s_{1}}-y_{s_{2}}$ et de remplacer le poids $x_{s_{1}}-y_{s_{2}}$ par $x_{s_{1}}$. Ce graphe serait un graphe de connexion pour la configuration initiale, ce qui est absurde. 

Ainsi, quitte à changer les signes, tout $s$-uplet non réalisable avec $s_{1}$ résidus positifs et~$s_{2}$ négatifs s'obtient à partir d'un $s-1$-uplet non réalisable $(x_{1},\dots,x_{s_{1}},-y_{1},\dots,-y_{s_{2}-1})$ auquel on ajoute un résidu $y_{s_{2}}$ et on remplace un résidu $x_{i}$ par $x_{i}+y_{s_{2}}$. On peut supposer, quitte à changer l'ordre, que $i=s_{1}$ et on note $x'_{s_{1}}:=x_{s_{1}}+y_{s_{2}}$. 
On cherche donc pour quelles valeurs de $y_{s_{2}}$ le $s$-uplet $(x_{1},\dots,x'_{s_{1}},-y_{1},\dots,-y_{s_{2}})$ est non réalisable. Par hypothèse de récurrence, on normalise ces nombres de telle sorte que les $x_{i}$ et les $y_{j}$ avec $j\neq s_{2}$ sont des entiers premiers entre eux. 
Si $y_{s_{2}}$ n'était pas un entier, alors ces résidus seraient dans l'image. En effet, un graphe de connexion serait obtenu de la façon suivante. On permute~$y_{1}$ et~$y_{s_{2}}$. On prend $s_{1}$ sommets en haut et $s_{2}$ sommets en bas. Le sommet $i_{0}$ en haut est relié au sommet $j_{0}$ en bas si et seulement si pour $J=j_{0}$ ou $J=j_{0}+1$, on a
 $$ \sum_{i\leq i_{0}-1}x_{i}\leq \sum_{j\leq J} y_{j} \leq  \sum_{i\leq i_{0}}x_{i}\, .$$
Les poids des sommets sont évidemment donnés par les $y_{j}$, les $x_{i}$ pour $i<s_{1}$ et $x'_{s_{1}}$ pour $i=s_{1}$.
Par conséquent, pour obtenir une configuration non réalisable, il est nécessaire que~$y_{s_{2}}$ soit un entier. En effet, dans le cas contraire, les sommes partielles ne peuvent pas coïncider car l'une des familles de sommes partielles n'est pas constituée d'entiers et donc l'opération consistant à retirer une feuille au graphe peut toujours s'effectuer. 
% Nous avons donc démontré par récurrence que les éléments des $s$-uplets non réalisables sont commensurables entre eux.
\end{proof}

Nous pouvons donc maintenant procéder à la preuve du point ii) du théorème~\ref{thm:geq0keq1} dans le cas des strates $\omoduli[0](s-2;(-1^{s}))$. Concrètement, nous montrons que l'image de l'application résiduelle de $\omoduli[0](s-2;(-1^{s}))$ avec $s\geq2$ est le complémentaire de l'union des plans $\CC^{\ast} \cdot (x_{1},\dots,x_{s_{1}},-y_{1},\dots,-y_{s_{2}})$ où $x_{i},y_{j} \in \NN$ sont premiers entre eux et  
\begin{equation*}
 \sum_{i=1}^{s_{1}} x_{i} = \sum_{j=1}^{s_{2}} y_{j} \leq s-2\,.
\end{equation*}
 \begin{proof}
 \'Etant donné un élément qui n’appartient pas à l'image de l'application résiduelle de $\omoduli[0](s-2;(-1^{s}))$.  Par le lemme~\ref{lem:finietcommens} on peut supposer que cet élément est de la forme $(x_{1},\dots,x_{s_{1}},-y_{1},\dots,-y_{s_{2}})$ avec  $x_{i}$ et $y_{j}$ des entiers premiers entre eux. Nous montrons  par récurrence que ce $s$-uplet vérifie l'équation~\eqref{eq:polessimples}.  
  Rappelons que cette équation est l'inégalité $\sum_{i=1}^{s_{1}} x_{i} = \sum_{j=1}^{s_{2}} y_{j} \leq s-2$.
 
 Nous commençons par montrer qu'il existe au moins quatre éléments du $s$-uplet qui sont égaux à $\pm1$. Cela est vrai dans le cas des strates $\omoduli[0](s-2;(-1^{s}))$ avec $s\leq 4$ puisque qu'on vérifie sans difficultés grâce au lemme~\ref{lem:gzeropolesimples} que le seul élément qui n'est pas dans l'image est $(1,1,-1,-1)$. Par récurrence, on considère un $s$-uplet $r := (x_{1},\dots,x_{s_{1}},-y_{1},\dots,-y_{s_{2}})$ qui n'est pas dans l'image de l'application résiduelle. Si tous les $x_{i},y_{j}$ sont égaux entre eux, le résultat est clair. Sinon on peut supposer que $\min(x_{i}) = x_{s_{1}} < \max(y_{j}) = y_{s_{2}}$. Le $s-1$-uplet $r' :=(x_{1},\dots,x_{s_{1}-1},-y_{1},\dots,-y_{s_{2}-1},x_{s_{1}}-y_{s_{2}})$ n'est pas réalisable. En effet, si c'était le cas il existerait un graphe de connexion associé à ces résidus On obtiendrait un graphe de connexion pour $r$ en collant une feuille de poids $x_{s_{1}}$ au sommet de poids $x_{s_{1}}-y_{s_{2}}$. Si le pgcd de $r'$ est $d > 1$, alors $r$ est de la forme $(dk_{1},\dots,dk_{s_{1}-1},x,-dl_{1},-dl_{s_{2}-1},-y)$ avec $x$ et $y$ premiers avec $d$. Mais alors il existe une différentielle qui possède ces résidus obtenue en collant sur un segment les cylindres de circonférence $(x,dk_{1},\dots,dk_{s_{1}-1})$ (dans cet ordre) en haut et $(-dl_{1},-dl_{s_{2}-1},-y)$ en bas. Par conséquent $d=1$. 
 
 En supposant (sans perte de généralité) que $x_{s_{1}}=\min(x_{i},y_{i})$, on obtient, par hypothèse de récurrence, que trois nombres parmi $x_{1},\dots,x_{s_{1}-1},y_{1},\dots,y_{s_{2}-1}$ sont égaux à $1$. Ajoutant $x_{s_{1}}$ (qui vaut $1$ par hypothèse) à cette liste, on obtient que $r$ possède au moins quatre éléments égaux à $\pm 1$.

 L'équation~\eqref{eq:polessimples} est clairement satisfaite pour l'unique élément $(1,1,-1,-1)$ qui n'est pas dans l'application résiduelle de $\omoduli[0](2;(-1^{4}))$. Supposons maintenant par récurrence qu'elle est satisfaite pour tous les $t$-uplet qui ne sont pas dans l'application résiduelle des strates $\omoduli[0](t-2;(-1^{t})$ avec $4 \leq t<s$. Considérons  l'élément $(x_{1},\dots,x_{s_{1}},-y_{1},\dots,-y_{s_{2}})$ qui n'est pas dans l'image de l'application résiduelle de $\omoduli[0](s-2;(-1^{s}))$.
 Si les résidus sont tous égaux entre eux, la borne est clairement satisfaite. Sinon, comme il existe un élément  du $s$-uplet égal à $1$, on peut supposer que $x_{s_{1}}>1$ et $y_{s_{2}}=1$. En enlevant le résidu $-y_{s_{2}}$ du $s$-uplet et en l'ajoutant à~$x_{s_{1}}$, on obtient un $(s-1)$-uplet qui n'est pas dans l'application résiduelle. Par récurrence, il vérifie l'équation~\eqref{eq:polessimples} avec $s-1$ termes. Comme $y_{s_{2}} = 1$, on a
$$\sum_{j=1}^{s_{2}} y_{j} = 1 + \sum_{j=1}^{s_{2}-1} y_{j} \leq  1 + (s-1)-2 = s-2\,.$$
Ceci démontre par récurrence la borne de l'équation~\eqref{eq:polessimples}.
 \smallskip
 \par
  Réciproquement, soit $(x_{1},\dots,x_{s_{1}},-y_{1},\dots,-y_{s_{2}})$  un élément avec  $x_{i}$ et $y_{j}$  des entiers premiers entre eux. S'il existe une différentielle $\omega$ de $\omoduli[0](s-2;(-1^{s-2}))$ qui possède ces résidus, alors  les longueurs de liens-selles de $\omega$ sont entières.    Comme il y a $s-1$ liens-selles, la somme des $x_{i}$ est supérieure ou égal à $s-1$ et donc l'équation~\eqref{eq:polessimples} n'est pas satisfaite. 
 \end{proof}

En utilisant cette caractérisation nous présentons dans la table~\ref{tab:noim} les éléments de l'espace résiduel qui ne sont pas dans l'image de l'application résiduelle des strates $\omoduli[0](s-2;(-1^{s}))$ pour~$s\leq6$.
\par
 \begin{table}[h]
\begin{tabular}{|c|c|}
  \hline
Strate & Éléments  pas dans l'image de l'application résiduelle (modulo $\CC^{\ast}$))  
\\
  \hline
  $\omoduli[0](0;(-1^{2}))$ & $\emptyset$  \\
  \hline
  $\omoduli[0](1;(-1^{3}))$ & $\emptyset$  \\
  \hline
 $\omoduli[0](2;(-1^{4}))$ & $(1,1,-1,-1)$  \\
  \hline
   $\omoduli[0](3;(-1^{5}))$ & $(2,1,-1,-1,-1)$  \\
  \hline
$\omoduli[0](4;(-1^{6}))$& $(1,1,1,-1,-1,-1), (2,1,1,-2,-1,-1), (2,2,-1,-1,-1,-1), (3,1,-1,-1,-1,-1)$  \\
  \hline
\end{tabular}
\caption{Les éléments qui ne sont pas des résidus de différentielles des strates $\omoduli[0](s-2;(-1^{s})$ avec $s\leq6$.}
\label{tab:noim}
\end{table}
 \par

Nous passons maintenant au cas des strates possédant au moins deux zéros. Elle s'obtient en utilisant le résultat pour les strates avec un unique zéro. Afin de décrire l'image de l'application résiduelle, nous faisons appel à la notion de différentielle stable rappelée dans la section~\ref{sec:pluridiffentre}.

\begin{lem}\label{lm:g0p-1plusieurszero}
 Soit $\omoduli[0](a_{1},\dots,a_{n};(-1^{s}))$ une strate de genre zéro avec $s$ pôles simples et~$n\geq2$ zéros. Le $s$-uplet $(r_{1},\ldots,r_{s})$ est dans l'image de l'application résiduelle si et seulement s'il existe une différentielle stable de genre zéro $(X,\omega)$ telle que:
 \begin{enumerate}
  \item la restriction de $\omega$ à chaque composante irreductible de $X$ possède un pôle simple aux points nodaux;
  \item la restriction de $\omega$ à la partie lisse de la courbe stable $X$ possède des singularités d'ordres  $(a_{1},\dots,a_{n};(-1^{s}))$ et les résidus aux pôles sont $(r_{1},\dots,r_{s})$;
  \item la restriction de $\omega$  à chaque composante irréductible de $X$ possède un unique zéro. 
 \end{enumerate}
\end{lem}

\begin{proof}
Soit  $r:=(r_{1},\dots,r_{s})\in\espres[0](\mu)$ un élément de l'espace résiduel de la strate $\omoduli[0](a_{1},\dots,a_{n};(-1^{s}))$. 
Supposons que $r$ soit dans l'image de l'application résiduelle de cette strate. Montrons l'existence d'une différentielle stable vérifiant les conditions du lemme~\ref{lm:g0p-1plusieurszero}. Soit $\omega$ une différentielle de $\omoduli[0](\mu)$ ayant pour résidus~$r$. Quitte à perturber $\omega$ sans changer les résidus, on peut supposer qu'il n'existe pas de liens selles horizontaux entre deux singularités coniques distinctes. En effet, le lieu de la strate $\omoduli[0](\mu)$ où les résidus sont $r$ est une variété orbifold de dimension~$n-1$ munie d'un atlas où les coordonnées sont les périodes des cycles de l'homologie relative (voir \cite{BCGGM3}). Comme tous les résidus sont réels, toute demi-droite horizontale issue d'une singularité conique heurte cette même singularité en temps fini. Coupons la surface plate associée à $\omega$ le long de ces liens selles. On obtient une union disjointe de cylindres et de demi-cylindres infinis.

La hauteur des cylindres d'aire finie peut être choisie arbitrairement sans changer la strate et les résidus. En faisant tendre toutes les hauteurs de ces cylindres vers l'infini, on obtient une différentielle stable. De plus, comme chaque zéro est relié à un autre par un cylindre, il y a précisément un zéro sur chaque composante irréductible de cette différentielle.
\smallskip
\par
L'autre implication est claire une conséquence directe du lemme~\ref{lem:lisspolessimples}. En effet, on obtient une différentielle dans la strate $\omoduli[0](\mu)$ ayant les résidus~$r$ par le lissage de la différentielle stable satisfaisant les conditions du lemme~\ref{lm:g0p-1plusieurszero}.
\end{proof}

Nous sommes maintenant en mesure de prouver le point ii) du théorème~\ref{thm:geq0keq1} pour toutes les strates avec $n\geq 2$ zéros.

Il s'agit de montrer que l'image de l'application résiduelle de $\omoduli[0](a_{1},\dots,a_{n};(-1^{s}))$ avec $s\geq2$ est le complémentaire de l'union des plans $\CC^{\ast} \cdot (x_{1},\dots,x_{s_{1}},-y_{1},\dots,-y_{s_{2}})$ où $x_{i},y_{j} \in \NN$ sont premiers entre eux et  
\begin{equation*}
 \sum_{i=1}^{s_{1}} x_{i} = \sum_{j=1}^{s_{2}} y_{j} \leq \max(a_{1},\dots,a_{n})\,.
\end{equation*}

\begin{proof}
 Considérons tout d'abord un $s$-uplet $(x_{1},\dots,x_{s_{1}},-y_{1},\dots,-y_{s_{2}})$ avec  $x_{i}$ et $y_{j}$ des entiers premiers entre eux tels que $T:=\sum_{i=1}^{s_{1}} x_{i} = \sum_{j=1}^{s_{2}} y_{j} > \max(a_{1},\dots,a_{n})$.   Si $n=1$, on a déjà montré qu'il existe une différentielle avec ces invariants locaux. Cela implique que si la somme $T$ est strictement supérieure à $s-2$, alors il suffit d'éclater le zéro d'une différentielle de $\omoduli[0](s-2;(-1^{s}))$ qui possède ces résidus.  On considère à partir de maintenant les $s$-uplets tels que $\max(a_{i}) < T \leq  s-2$. 
 
 Si $s_{1}=s_{2}=s/2-1$, alors ce uplet est réalisable dans la strate $\omoduli[0](s/2-1, s/2-1;(-1^{s}))$. En effet il suffit  de lisser la différentielle stable formée de la façon suivante. Elle possède deux composantes irréductibles $X_{1}$ et~$X_{2}$ avec un nœud entre les deux. Chaque composante contient un zéro, la composante $X_{1}$ contient tous les résidus positifs et $X_{2}$ les résidus négatifs. Le résidu aux points nodaux est l'opposé des sommes des résidus.

On considère maintenant les strates distinctes de $\omoduli[0](s/2-1, s/2-1;(-1^{s}))$. Nous supposons par récurrence que le résultat est vrai pour $n-1$. Comme nous avons l'inégalité $\min(a_{i})+\max(a_{i}) \leq s-2$, il s'ensuit que $\min(1+a_{i})+\max(1+a_{i}) \leq \min(s_{1},s_{2})+\max(s_{1},s_{2})$. On en déduit l'inégalité $1 + \min(a_{i}) \leq \max (s_{1},s_{2})$. Cette inégalité est stricte sauf si $a_{1}=a_{2}=s/2-1$, que nous ne considérons pas ici. Sans perte de généralité, on suppose que $a_{1}$ est le zéro d’ordre minimal et que $s_{1}\geq s_{2}$. C'est donc l'inégalité $1 + a_1 < s_{1}$ qui est satisfaite.

Comme $T \leq s-2$, la configuration n'est pas réalisable dans la strate $\omoduli[0](s-2;(-1^{s}))$. Il est démontré dans la preuve du théorème~\ref{thm:geq0keq1} pour le cas $n=1$ qu'une telle configuration compte au moins quatre résidus de valeur absolue égale à~$1$.

Nous allons considérer la différentielle stable sur la courbe stable formée de deux composantes $X_{1}$ et~$X_{2}$ reliées par un unique nœud construite de la façon suivante. Choisissons un ensemble~$\mathcal{P}$ de $1 + a_1$ éléments~$x_{i}$ tel que, si tous les éléments de norme $1$ sont positifs, alors le complémentaire de $\mathcal{P}$ contient au moins un résidu de valeur absolue $1$. Sur  $X_1$, on met le zéro d’ordre $a_1$ et les poles simples dont les résidus sont dans~$\mathcal{P}$. Le résidu au point nodal est l'opposé de la somme de ces résidus. Une telle différentielle existe car elle possède un unique pôle simple dont
le résidu est  négatif. Sur la composante $X_2$, on met tous les autres zéros et pôles. Par construction, il existe au moins un pôle dont le résidu est de valeur absolue $1$. Donc le pgcd des résidus des pôles de cette composante est égal à $1$. De plus, la somme des résidus négatifs sur $X_{2}$ est égale à $T$, qui est par hypothèse strictement
supérieure à l'ordre du zéro maximal de $X_{2}$. Comme cet ordre coïncide avec le maximum sur tous les ordres, l’hypothèse de récurrence donne une différentielle sur $X_{2}$ possédant ces invariants. On obtient la différentielle souhaitée en lissant cette différentielle stable.
%  
%  Considérons la différentielle stable formée de la forme suivante. Elle possède deux composantes irréductibles $X_{1}$ et~$X_{2}$ avec un nœud entre les deux. La composante $X_{1}$ qui contient le zéro d'ordre maximal, nous supposerons qu'il s'agit de  $a_{n}$, et la composante $X_{2}$ qui contient les zéros $a_{1},\dots,a_{n-1}$. 
%  
%  Si $a_{n} \geq s_{1}+2$, alors la différentielle sur $X_{1}$ contient tous les pôles dont les résidus sont $x_{i}$. Les pôles dont les résidus sont $y_{j}$ sont répartis entre les deux composantes $X_{1}$ et~$X_{2}$.  Montrons qu'il existe  une différentielle stable sur cette courbe stable satisfaisant les conditions du lemme~\ref{lm:g0p-1plusieurszero}. Sur la composante $X_{1}$, on a $\sum x_{i} > a_{n}$, donc le cas $n=1$ donne l’existence d'une telle différentielle. Sur~$X_{2}$, on a un unique pôle avec un résidu positif (au point nodal). Cela donne l'existence de la différentielle sur la composante~$X_{2}$.
%  
%  Si $a_{n} < s_{1}+2$, alors on met $a_{n}+2$ pôles dont les résidus sont positifs sur la composante~$X_{1}$ et les autres pôles sur la composante $X_{2}$. Sur $X_{1}$, l'existence de la différentielle suit du fait qu'il n'y a qu'un pôle de résidu négatif. Sur $X_{2}$, on a $\sum y_{j} > a_{n} \geq \max(a_{1},\dots,a_{n-1})$ et donc l'existence de la différentielle possédant les invariants souhaités est donnée par récurrence.
 \smallskip
 \par
Réciproquement, on considère un $s$-uplet avec $\sum_{i=1}^{s_{1}} x_{i} = \sum_{j=1}^{s_{2}} y_{j} \leq \max(a_{1},\dots,a_{n})$. Nous supposons que ce $s$-uplet est réalisé par une différentielle telle que $n$ est minimal. Notons que nous avons montré que $n\geq2$ dans la première partie de cette section. Sans perte de généralité, on supposera que $a_{n}= \max(a_{1},\dots,a_{n})$.

On considère une différentielle stable satisfaisant les conditions du lemme~\ref{lm:g0p-1plusieurszero}. Considérons une composante $X_{1}$ qui est reliée au reste de la courbe par un unique nœud, i.e. cette composante correspond à une feuille du graphe dual. Comme ce graphe contient au moins deux feuilles, nous supposerons que $X_{1}$ ne contient pas $a_{n}$.  

Supposons que le résidu de la restriction de la différentielle à $X_{1}$ possède un résidu positif au point nodal. Par stabilité, la restriction de la différentielle au complémentaire~$\tilde X$ de $X_{1}$ possède un pôle avec un résidu négatif au point nodal correspondant. Cette restriction satisfait aux conditions du lemme~\ref{lm:g0p-1plusieurszero} (en considérant le point nodal avec $X_{1}$ comme un pôle de la partie lisse). Supposons que les $\tilde{s_{1}}$ premiers résidus positifs appartiennent à $\tilde X$, on a alors $\sum_{i=1}^{\tilde{s_{1}}} x_{i} \leq \sum_{i=1}^{s_{1}} x_{i} \leq a_{n}$. En posant $d$ le pgcd des résidus de la partie lisse de $\tilde{X}$, cette inégalité implique que  $\sum_{i=1}^{\tilde{s_{1}}} \tfrac{x_{i}}{d} \leq \sum_{i=1}^{s_{1}} \tfrac{x_{i}}{d} \leq a_{n}$. Comme il y a $n-1$ zéros sur $\tilde X$, cela contredit la minimalité de notre exemple.
 \end{proof}

Nous terminons cette section par un exemple qui illustre les construction du lemme~\ref{lm:g0p-1plusieurszero}.
\begin{ex}\label{ex:graphedeliant}
Dans cet exemple, nous donnons des différentielles dans  $\omoduli[0](1,3;(-1^{6}))$ et $\omoduli[0](2,2;(-1^{6}))$ dont les r\'esidus sont $(2,1,1,-1,-1,-2)$ dans la figure~\ref{fig:graphdeliant}.  Les graphes duaux des différentielles stables que l'on obtient en faisant tendre les hauteurs des cylindres fini vers l'infini. Ces différentielles stables sont celles données dans le lemme~\ref{lm:g0p-1plusieurszero}. Les pôles simples correspondent aux demi-arêtes  et les labels aux résidus de ces pôles. On peut remarquer que ces r\'esidus ne sont pas dans l'image de $\appres[0](4;(-1^{6}))$.

\begin{figure}[htb]
\center
\begin{tikzpicture}[scale=1.2]
    
    \begin{scope}[xshift=-4cm]
%Curve
\fill[fill=black!10] (-4,-1.9)  -- (-2,-1.9)-- (-2,.9)-- (-4,.9) -- cycle;

   \foreach \i in {1,2,...,5}
  \coordinate (a\i) at (-1.5-\i/2,-1); 
       \foreach \i in {1,2,3,5}
   \fill (a\i)  circle (2pt);

     \foreach \i in {1,2,3,5}
     \draw (a\i) -- ++(0,-1);
     
   \foreach \i in {1,2,...,5}
  \coordinate (b\i) at (-1.5-\i/2,0); 
      
        \foreach \i in {1,2,3,5}
     \draw (b\i) -- ++(0,1);
     
     \draw (a1) -- (b1);
     \draw (a5) -- (b5);

      \foreach \i in {1,2,3,5}
   \fill[white] (b\i)  circle (2pt);
          \foreach \i in {1,2,3,5}
     \draw (b\i)  circle (2pt);
     
     \node[yshift=-.6cm, xshift=.1cm] at (a5) {$1$};
     \node[yshift=-.6cm, xshift=-.1cm] at (a3) {$1$};
      \node[yshift=-.6cm, xshift=.1cm] at (a3) {$2$};
     \node[yshift=-.6cm, xshift=-.1cm] at (a2) {$2$};
      \node[yshift=-.6cm, xshift=.1cm] at (a2) {$3$};
     \node[yshift=-.6cm, xshift=-.1cm] at (a1) {$3$};

     \node[yshift=.6cm, xshift=.1cm] at (b5) {$4$};
     \node[yshift=.6cm, xshift=-.1cm] at (b3) {$4$};
      \node[yshift=.6cm, xshift=.1cm] at (b3) {$5$};
     \node[yshift=.6cm, xshift=-.1cm] at (b2) {$5$};
      \node[yshift=.6cm, xshift=.1cm] at (b2) {$6$};
     \node[yshift=.6cm, xshift=-.1cm] at (b1) {$6$};
      
   \node[yshift=.5cm, xshift=.1cm] at (a5) {$0$};    
   \node[yshift=.5cm, xshift=-.1cm] at (a1) {$0$};    

%Ordre
\coordinate  (x1)  at  (0,0);
\fill (0,-1) coordinate (x2) circle (2pt); 
\draw (x1) -- ++(.5,0) coordinate (r1);
\draw (x1) -- ++(0,.5) coordinate (r2);
\draw (x1) -- ++(-.5,0) coordinate (r3);
\node[right] at (r1) {$1$};
\node[above] at (r2) {$1$};
\node[left] at (r3) {$2$};

\draw[] (x1) -- (x2) coordinate[pos=.3](R1)coordinate[pos=.7](R2);
\node[left] at (R2) {$4$};\node[left] at (R1) {$-4$};

\draw (x2) -- ++(.5,0) coordinate (r4);
\draw (x2) -- ++(0,-.5) coordinate (r5);
\draw (x2) -- ++(-.5,0) coordinate (r6);
\node[right] at (r4) {$-1$};
\node[below] at (r5) {$-1$};
\node[left] at (r6) {$-2$};

 \fill[white] (x1)  circle (2pt);
     \draw (x1)  circle (2pt);
\end{scope}

   \begin{scope}[xshift=3cm]
%Curve
\fill[fill=black!10] (-4,-1.9)  -- (-2,-1.9)-- (-2,-1)-- (-1.8,0)--(-2.8,0) -- (-3,-1) -- (-3,.9) -- (-4,.9) -- cycle;

   \foreach \i in {1,2,...,5}
  \coordinate (a\i) at (-1.5-\i/2,-1); 
       \foreach \i in {1,2,...,5}
   \fill (a\i)  circle (2pt);

     \foreach \i in {1,2,4,5}
     \draw (a\i) -- ++(0,-1);
     \draw (a1) -- ++(80:1);
     \draw (a3) -- ++(80:1);
     
   \foreach \i in {3,4,5}
  \coordinate (b\i) at (-1.5-\i/2,0); 

        \foreach \i in {3,4,5}
     \draw (b\i) -- ++(0,1);
     
     \draw (a3) -- (b3);
     \draw (a5) -- (b5);

          \foreach \i in {3,4,5}
   \fill[white] (b\i)  circle (2pt);
          \foreach \i in {3,4,5}
     \draw (b\i)  circle (2pt);  
     
        \node[yshift=-.6cm, xshift=.1cm] at (a5) {$1$};
     \node[yshift=-.6cm, xshift=-.1cm] at (a4) {$1$};
      \node[yshift=-.6cm, xshift=.1cm] at (a4) {$2$};
     \node[yshift=-.6cm, xshift=-.1cm] at (a2) {$2$};
      \node[yshift=-.6cm, xshift=.1cm] at (a2) {$3$};
     \node[yshift=-.6cm, xshift=-.1cm] at (a1) {$3$};

     \node[yshift=.6cm, xshift=.1cm] at (b5) {$4$};
     \node[yshift=.6cm, xshift=-.1cm] at (b4) {$4$};
      \node[yshift=.6cm, xshift=.1cm] at (b4) {$5$};
     \node[yshift=.6cm, xshift=-.1cm] at (b3) {$5$};
      \node[yshift=.6cm, xshift=.25cm] at (a3) {$6$};
     \node[yshift=.6cm, xshift=0cm] at (a1) {$6$};
      
   \node[yshift=.5cm, xshift=.1cm] at (a5) {$0$};    
   \node[yshift=.5cm, xshift=-.1cm] at (a3) {$0$};     
     
%Ordre
\coordinate  (x1)  at  (0,0);
\fill (0,-1) coordinate (x2) circle (2pt); 
\draw (x1) -- ++(45:.5) coordinate (r1);
\draw (x1) -- ++(135:.5) coordinate (r2);
\node[above] at (r1) {$1$};
\node[above] at (r2) {$1$};
 
\draw[] (x1) -- (x2)  coordinate[pos=.3](R1)coordinate[pos=.7](R2);
\node[left] at (R2) {$2$};\node[left] at (R1) {$-2$};

\draw (x2) -- ++(0:.5) coordinate (r3);
\draw (x2) -- ++(-60:.5) coordinate (r4);
\draw (x2) -- ++(-120:.5) coordinate (r5);
\draw (x2) -- ++(-180:.5) coordinate (r6);
\node[right] at (r3) {$2$};
\node[below] at (r4) {$-1$};
\node[below left] at (r5) {$-2$};
\node[left] at (r6) {$-1$};

 \fill[white] (x1)  circle (2pt);
     \draw (x1)  circle (2pt);
     
\end{scope}
\end{tikzpicture}
\caption{Différentielles dans $\omoduli[0](2,2;(-1^6))$ et $\omoduli[0](1,3;(-1^6))$ avec résidus $(2,1,1,-1,-1,-2)$ et leurs graphes duaux des différentielles stables associées.} \label{fig:graphdeliant}
\end{figure}
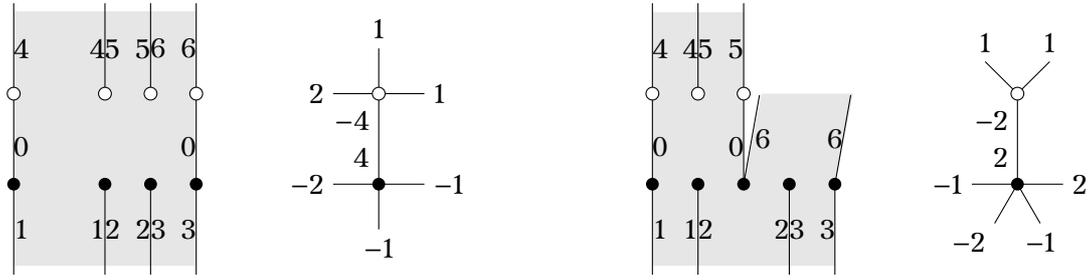

\end{ex}

\section{Les différentielles sur les surfaces de genre supérieur}
\label{sec:ggeq1}

Dans cette section nous prouvons le théorème~\ref{thm:ggeq1abel} qui énonce la surjectivité de l'application résiduelle de chaque composante connexe en genre $g\geq1$.
Nous allons tout d'abord traiter le cas des différentielles sur les surfaces de Riemann de genre $1$ dans la section~\ref{sec:g1}. Le cas des strates de genre supérieur ou égal à $2$ est donné dans la section~\ref{sec:ggeq2}.

\subsection{Différentielles en genre $1$}
\label{sec:g1}

En genre un, Boissy à montré dans la section~4.2 de \cite{Bo} que les composantes connexes des strates sont caractérisées par le nombre de rotation $\rot(S)$ des surfaces de translations associées. Pour une surface de translation $S$ définie par une différentielle méromorphe de $\omoduli[1](a_{1},\dots,a_{n};-b_{1},\dots,-b_{p})$ avec une base symplectique de lacets lisses de l'homologie $(\alpha,\beta)$ le {\em nombre de rotation} est $$\rot(S):=\pgcd(a_{1},\dots,a_{n};b_{1},\dots,b_{p},\ind(\alpha),\ind(\beta))\,,$$ où $\ind(\alpha)$ est l'indice de l’application de Gauss du lacet $\alpha$. On a le résultat suivant dû à Boissy.
\begin{itemize}
 \item[i)] Si $n=p=1$, la strate est $\omoduli[1](a;-a)$ avec $a\geq2$ et chaque composante connexe correspond à un nombre de rotation qui est un diviseur strict de $a$.
 \item[ii)] Sinon, il existe une composante connexe correspondant à chaque nombre de rotation qui est un diviseur de $\pgcd(a_{1},\dots,a_{n};b_{1},\dots,b_{p})$.
\end{itemize}

\medskip
\par
Nous commençons par montrer la surjectivité de l'application résiduelle $\appres[1](s;(-1^{s}))$ avec $s\geq2$. Par le résultat précédent, ces strates sont connexes. 
\begin{lem}\label{lem:gunspe1}
 L'application résiduelle de la strate $\omoduli[1](s;(-1^{s}))$ est surjective pour $s>1$.
\end{lem}
\begin{proof}
Considérons la strate $\omoduli[1](s;(-1^{s}))$, avec $s>1$  et $r:=(r_{1},\dots,r_{s})$ un élément de l'espace résiduel $\espres[1](s;(-1^{s}))$. Prenons un tore plat $S_{1}$ tel que le lien selle le plus petit est strictement supérieur à $\sum|r_{i}|$. Nous enlevons de $S_{1}$ le polygone résiduel $\mathfrak{P}(-r)$. Cette opération est réalisable par notre hypothèse sur la longueur des liens selles de $S_{1}$.
Pour chacun des pôles~$P_{i}$, nous prenons une partie polaire d'ordre~$1$ associée à~$r_{i}$. Nous collons le bord de ces parties polaires au bord de $S_{1}\setminus \mathfrak{P}(r)$ par translation. La construction est représentée par le dessin de la figure~\ref{fig:casgeneralgenreun}. On vérifie sans problème que la surface ainsi obtenue est de genre $1$ et possède une unique singularité conique.
 \begin{figure}[htb]
 \center
\begin{tikzpicture}[scale=1.6]

%pole 1,1
\begin{scope}[xshift=-7.5cm,yshift=-.8cm]
\coordinate (a) at (-1,1);
\coordinate (b) at (0,1);

    \fill[fill=black!10] (a)  -- (b)coordinate[pos=.5](f) -- ++(0,1) --++(-1,0) -- cycle;
    \fill (a)  circle (2pt);
\fill[] (b) circle (2pt);
 \draw  (a) -- (b);
 \draw (a) -- ++(0,.9) coordinate (d)coordinate[pos=.5](h);
 \draw (b) -- ++(0,.9) coordinate (e)coordinate[pos=.5](i);
 \draw[dotted] (d) -- ++(0,.2);
 \draw[dotted] (e) -- ++(0,.2);
\node[above] at (f) {$1$};
\end{scope}

%pole 1,2
\begin{scope}[xshift=-2.5cm,yshift=-.8cm]
\coordinate (a) at (-1,1);
\coordinate (b) at (0,1);

    \fill[fill=black!10] (a)  -- (b)coordinate[pos=.5](f) -- ++(0,1) --++(-1,0) -- cycle;
    \fill (a)  circle (2pt);
\fill[] (b) circle (2pt);
 \draw  (a) -- (b);
 \draw (a) -- ++(0,.9) coordinate (d)coordinate[pos=.5](h);
 \draw (b) -- ++(0,.9) coordinate (e)coordinate[pos=.5](i);
 \draw[dotted] (d) -- ++(0,.2);
 \draw[dotted] (e) -- ++(0,.2);
\node[above] at (f) {$2$};
\end{scope}

%pole 1,3
\begin{scope}[xshift=-7.5cm,yshift=-1.2cm]
\coordinate (a) at (-1,1);
\coordinate (b) at (0,1);

    \fill[fill=black!10] (a)  -- (b)coordinate[pos=.5](f) -- ++(0,-1) --++(-1,0) -- cycle;
    \fill (a)  circle (2pt);
\fill[] (b) circle (2pt);
 \draw  (a) -- (b);
 \draw (a) -- ++(0,-.9) coordinate (d)coordinate[pos=.5](h);
 \draw (b) -- ++(0,-.9) coordinate (e)coordinate[pos=.5](i);
 \draw[dotted] (d) -- ++(0,-.2);
 \draw[dotted] (e) -- ++(0,-.2);
\node[below] at (f) {$3$};
\end{scope}

%pole 1,4
\begin{scope}[xshift=-2.5cm,yshift=-1.2cm]
\coordinate (a) at (-1,1);
\coordinate (b) at (0,1);

    \fill[fill=black!10] (a)  -- (b)coordinate[pos=.5](f) -- ++(0,-1) --++(-1,0) -- cycle;
    \fill (a)  circle (2pt);
\fill[] (b) circle (2pt);
 \draw  (a) -- (b);
 \draw (a) -- ++(0,-.9) coordinate (d)coordinate[pos=.5](h);
 \draw (b) -- ++(0,-.9) coordinate (e)coordinate[pos=.5](i);
 \draw[dotted] (d) -- ++(0,-.2);
 \draw[dotted] (e) -- ++(0,-.2);
\node[below] at (f) {$4$};
\end{scope}

%Tore
\begin{scope}[xshift=-7cm,yshift=-.5cm]
\coordinate (a) at (0,0);
\coordinate (b) at (3,0);
\coordinate (c) at (3,1);
\coordinate (d) at (0,1);

    \fill[fill=black!10] (a)  -- (b) -- (c) -- (d) -- cycle;
 \draw  (a) -- (b) --(c)--(d)--(a);

\draw (.5,.5) coordinate (e) --(1.5,.5) coordinate (f)coordinate[pos=.5](h)  --(2.5,.5) coordinate (g)coordinate[pos=.5](i);
\fill (e)  circle (2pt);
\fill (f) circle (2pt);
\fill (g) circle (2pt);
\node[below] at (h) {$1$};
\node[below] at (i) {$2$};
\node[above] at (h) {$3$};
\node[above] at (i) {$4$};
\end{scope}
\end{tikzpicture}
\caption{Une différentielle de $\omoduli[1](4;(-1^{4}))$ dont les résidus sont $(1,1,-1,-1)$.} \label{fig:casgeneralgenreun}
\end{figure}
\end{proof}

Nous montrons maintenant la surjectivité de l'application résiduelle restreinte à chaque composante connexe de genre supérieur ou égal à un. 
\begin{proof}[Preuve du théorème~\ref{thm:ggeq1abel} en genre $1$]
Nous fixons une strate $\omoduli[1](\mu)$ de genre $1$ avec  la partition $\mu:=(a_{1},\dots,a_{n};-b_{1},\dots,-b_{p};(-1^{s}))$ et un élément $r$ de l'espace résiduel $\espres[1](\mu)$.
Commençons par remarquer qu'il suffit de montrer le résultat pour les strates avec un unique zéro. En effet, la proposition~7.1 de \cite{Bo} implique que dans chaque composante connexe de strates avec $n\geq2$ zéros, il existe une différentielle obtenue en éclatant le zéro d'une différentielle avec un unique zéro. Le résultat se déduit car la proposition~\ref{prop:eclatZero} montre que cette opération ne modifie pas les résidus aux pôles.  Dans la suite nous supposerons que $n=1$.

Dans le cas où $s\neq0$, les strates sont connexes. Si $p = 0$, la surjectivité de l'application résiduelle a été montrée dans le lemme~\ref{lem:gunspe1}.
Dans le cas où~$p\neq0$, alors le théorème~\ref{thm:geq0keq1} montre qu'il existe une différentielle dans $\omoduli[0](a-2;-b_{1},\dots,-b_{p};(-1^{s}))$ dont les résidus sont $r$.  On obtient une différentielle dans $\omoduli[1](\mu)$ ayant pour résidus $r$ en cousant une anse à ces différentielles (voir la proposition~\ref{prop:attachanse}). 

\`A partir de maintenant, nous supposerons que $s=0$, i.e. que les strates paramètrent des différentielles dont les pôles sont d'ordres~$b_{i}\geq 2$. La proposition est trivialement vraie lorsqu'il y a un unique pôle. On supposera donc que $p\geq2$  dans la suite de la preuve. La proposition~6.1 de \cite{Bo} nous indique que chaque strate peut s'obtenir en cousant une anse à différentielles de genre~$0$. De plus, le théorème~\ref{thm:geq0keq1} indique qu'il existe des différentielles dans ces strates dont les résidus sont $(r_{1},\dots,r_{p})\in \espres[0](\mu)\setminus\left\{(0,\dots,0)\right\}$.  La couture d'anse à partir de ces strates permet d'obtenir une différentielle dans chaque composante connexe de $\omoduli[1](a;-b_{1},\dots,-b_{p})$ dont les résidus sont $(r_{1},\dots,r_{p}) \neq (0,\dots,0)$. Il reste donc à construire dans chaque composante connexe des strates $\omoduli[1](a;-b_{1},\dots,-b_{p})$ une différentielle dont tous les résidus sont nuls.

Dans ce cas, nous considérons la construction suivante. Pour tous les pôles, nous prenons une partie polaire de type $b_{i}$ associée aux vecteurs $(1;1)$. On colle le bord supérieur de $P_{i}$ au bord inférieur de $P_{i+1}$.
Il reste deux liens selles homologues que l'on relie par un cylindre. La surface obtenue possède les invariants locaux souhaités.

 \begin{figure}[htb]
 \center
\begin{tikzpicture}[scale=1.25, decoration={    markings,
    mark=at position 0.4 with {\arrow[very thick]{>}}}]
\begin{scope}[xshift=-6cm]
\fill[fill=black!10] (0,0) coordinate (Q) circle (1.1cm);

\draw[] (0,0) coordinate (Q) -- (1.1,0) coordinate[pos=.5](a);

\node[above] at (a) {$1$};
\node[below] at (a) {$2$};

\draw[] (Q) -- (-.8,0) coordinate (P) coordinate[pos=.5](c);

\fill (Q)  circle (1pt);

\fill (P) circle (1pt);
\node[above] at (c) {$a$};
\node[below] at (c) {$c$};

 \draw[postaction={decorate},red] (-.7,0) .. controls ++(90:.5)  and ++(90:.5) .. (-.9,0)  .. controls         ++(-90:.5) and ++(-90:.5) .. (-.7,0);
 \node[red] at (-1.1,0) {$\beta$};
\end{scope}

%deuxieme dessin
\begin{scope}[xshift=-3.7cm]
\fill[fill=black!10] (0,0) coordinate (Q) circle (1.1cm);

\draw[] (0,0) -- (1.1,0) coordinate[pos=.5](a);

\node[above] at (a) {$2$};
\node[below] at (a) {$1$};
\fill (Q)  circle (1pt);
\end{scope}

%dessin en dessous
\begin{scope}[xshift=1cm]
\fill[fill=black!10] (0,0)  circle (1.1cm);

\draw[] (0,0) coordinate (Q) -- (1.1,0) coordinate[pos=.5](a);

\node[above] at (a) {$1$};
\node[below] at (a) {$2$};

\draw[] (0,0) -- (-.8,0) coordinate (P) coordinate[pos=.5](c);

\fill (Q)  circle (1pt);

\fill (P) circle (1pt);
\node[above] at (c) {$a$};
\node[below] at (c) {$c$};

 \draw[postaction={decorate},red] (-.7,0) .. controls ++(90:.5)  and ++(90:.5) .. (-.9,0)  .. controls         ++(-90:.5) and ++(-90:.5) .. (-.7,0);
  \node[red] at (-1.1,0) {$\beta$};
\end{scope}

%Deuxieme dessin en dessous
\begin{scope}[xshift=3.3cm]
\fill[fill=black!10] (0,0) coordinate (Q) circle (1.1cm);
    
\draw[] (0,0) -- (1.1,0) coordinate[pos=.5](a);

\node[above] at (a) {$2$};
\node[below] at (a) {$1$};
\fill (Q)  circle (1pt);
\end{scope}

%%%%%%%%%%%%%%%%%%%%%%%%%%%% Seconde ligne %%%%%%%%%%%%%%%%
\begin{scope}[yshift=-2.7cm]
 \begin{scope}[xshift=-6cm]
\fill[fill=black!10] (0,0) coordinate (Q) circle (1.1cm);

\draw[] (0,0) coordinate (Q) -- (1.1,0) coordinate[pos=.5](a);

\node[above] at (a) {$3$};
\node[below] at (a) {$4$};

\draw[] (Q) -- (-.8,0) coordinate (P) coordinate[pos=.5](c);

\fill (Q)  circle (1pt);
%\fill (Q1)  arc [start angle=180,end angle=0,radius=2pt];
%\fill (Q2) arc [start angle=180,end angle=360,radius=2pt];

\fill (P) circle (1pt);
\node[above] at (c) {$c$};
\node[below] at (c) {$b$};

 \draw[postaction={decorate},red] (-.7,0) .. controls ++(90:.5)  and ++(90:.5) .. (-.9,0)  .. controls         ++(-90:.5) and ++(-90:.5) .. (-.7,0);
\end{scope}

%deuxieme dessin
\begin{scope}[xshift=-3.7cm]
\fill[fill=black!10] (0,0) coordinate (Q) circle (1.1cm);

\draw[] (0,0) -- (1.1,0) coordinate[pos=.5](a);

\node[above] at (a) {$4$};
\node[below] at (a) {$3$};
\fill (Q)  circle (1pt);
\end{scope}

%dessin en dessous
\begin{scope}[xshift=1cm]
\fill[fill=black!10] (0,0)  circle (1.1cm);

\draw[] (0,0) coordinate (Q) -- (-1.1,0) coordinate[pos=.5](a);

\node[above] at (a) {$4$};
\node[below] at (a) {$3$};

\draw[] (0,0) -- (.8,0) coordinate (P) coordinate[pos=.5](c);

\fill (Q)  circle (1pt);

\fill (P) circle (1pt);
\node[above] at (c) {$c$};
\node[below] at (c) {$b$};

 \draw[postaction={decorate},red] (.1,0) .. controls ++(90:.5)  and ++(90:.5) .. (-.1,0)  .. controls         ++(-90:.5) and ++(-90:.5) .. (.1,0);
\end{scope}

%Deuxieme dessin en dessous
\begin{scope}[xshift=3.3cm]
\fill[fill=black!10] (0,0) coordinate (Q) circle (1.1cm);
    
\draw[] (0,0) -- (-1.1,0) coordinate[pos=.5](a);

\node[above] at (a) {$3$};
\node[below] at (a) {$4$};
\fill (Q)  circle (1pt);

 \draw[postaction={decorate},red] (.1,0) .. controls ++(90:.3)  and ++(90:.3) .. (-.1,0)  .. controls         ++(-90:.3) and ++(-90:.3) .. (.1,0);
\end{scope}
\end{scope}

%%%%%%%%%%%%%%%% Les cylindres %%%%%%%%%%%%%%%%%%%%%%%

\begin{scope}[xshift=-5.25cm,yshift=-1.7cm]
\coordinate (a) at (0,0);
\coordinate (b) at (.8,0);
\coordinate (c) at (.8,.8);
\coordinate (d) at (0,.8);
\filldraw[fill=black!10] (a) -- (b)coordinate[pos=.5](h) -- (c) coordinate[pos=.5](e) -- (d) coordinate[pos=.5](f) -- (a) coordinate[pos=.5](g);
\fill (a) circle (1pt);\fill (b) circle (1pt);\fill (c) circle (1pt);\fill (d) circle (1pt);
\node[above] at (h) {$b$};
\node[below] at (f) {$a$};
\node[right] at (e) {$5$};
\node[left] at (g) {$5$};

\draw[postaction={decorate},red] (.1,0) -- (.1,.8);
\draw[postaction={decorate},blue] (0,.4) -- (.8,.4);
  \node[blue] at (.7,.5) {$\alpha$};
\end{scope}

\begin{scope}[xshift=1.75cm,yshift=-1.7cm]
\coordinate (a) at (0,0);
\coordinate (b) at (.8,0);
\coordinate (c) at (.8,.8);
\coordinate (d) at (0,.8);
\filldraw[fill=black!10] (a) -- (b)coordinate[pos=.5](h) -- (c) coordinate[pos=.5](e) -- (d) coordinate[pos=.5](f) -- (a) coordinate[pos=.5](g);
\fill (a) circle (1pt);\fill (b) circle (1pt);\fill (c) circle (1pt);\fill (d) circle (1pt);
\node[above] at (h) {$b$};
\node[below] at (f) {$a$};
\node[right] at (e) {$5$};
\node[left] at (g) {$5$};

\draw[postaction={decorate},red] (.1,0) -- (.1,.8);
\draw[postaction={decorate},blue] (0,.4) -- (.8,.4);
  \node[blue] at (.7,.5) {$\alpha$};
\end{scope}
\end{tikzpicture}
\caption{Une différentielle de la strate $\omoduli[1](6;-3,-3)$ dont les résidus sont $(0,0)$  et de nombre de rotation $1$ à gauche et $3$ à droite.} \label{fig:casgeneralgenreunbis}
\end{figure}
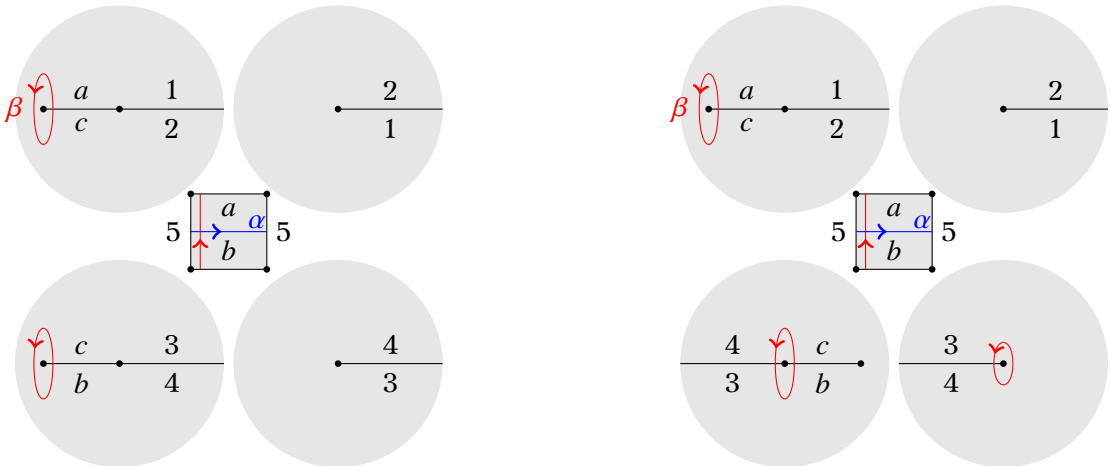

Une base de l'homologie est donnée par une géodésique périodique $\alpha$ du cylindre (donc d'indice zéro) et le lacet $\beta$ suivant. Il coupe $\alpha$ dans le cylindre puis le lien selle au bord du domaine polaire $P_{p}$, puis tourne à gauche avant de ressortir de ce domaine polaire en coupant l'autre lien selle et ainsi de suite. Le cas de la strate $\omoduli[1](6;-3,-3)$ est illustré dans la figure~\ref{fig:casgeneralgenreunbis}.
Remarquons que changer le type $\tau$ d'une partie polaire change d'autant l'indice de $\beta$. On peut donc obtenir pour $\beta$ tous les indices  dans $J=\left[ p,-p+\sum_{i=1}^{p} b_{i}\right]$.   \`A moins que la totalité des pôles ne soient d'ordre~$-2$, on obtient ainsi toutes les composantes connexes de la strate. En effet, si $\mu=(a;-3,-2,\dots,-2)$, alors la strate est connexe. Il suffit donc de montrer que la longueur de l'intervalle $J$ est supérieur ou égale à  $\min_{i} b_{i}$, si $(b_{1},\dots,b_{p})\neq (2,\dots,2)$ et $(b_{1},\dots,b_{p})\neq (3,2,\dots,2)$. Cette inégalité est clairement satisfaite en dehors du cas $p=2$ et $b_{1}=b_{2}=3$ (cas déjà considéré plus haut).

\smallskip
\par
Dans une strate minimale avec uniquement des pôles d'ordre~$-2$, il y a exactement deux composantes connexes. La construction qui précède ne permet d'obtenir que la composante dont le nombre de rotation a la même parité que $p$.
On propose alors une deuxième construction. On prend $p-1$ parties polaires et on les colle  comme précédemment. La dernière partie polaire est associée aux vecteurs $(i,1;1,i)$.
On identifie les vecteurs $i$ entre eux et les deux autres bords comme précédemment. Cette construction est illustrée dans la figure~\ref{fig:casgeneralgenreunter} dans le cas de la strate $\omoduli[1](4;-2,-2)$.
Le lacet $\beta$ est défini comme précédemment et  a pour indice~$p$. Le lacet $\alpha$ connecte le milieu de des segments $v$ sans sortir du domaine polaire de~$P_{p}$. Son indice est donc~$1$. On construit une différentielle avec n'importe quels résidus dans la composante connexe dont le nombre de rotation est $1$.

\begin{figure}[htb]
 \center
\begin{tikzpicture}[scale=1.5, decoration={    markings,
    mark=at position 0.5 with {\arrow[very thick]{>}}}]
%pole2,1
\begin{scope}[xshift=1cm,yshift=0cm]
\fill[fill=black!10] (0,0) coordinate (Q) circle (1.3cm);
\coordinate (a) at (-.5,-.5);
\coordinate (b) at (.5,-.5);
\coordinate (c) at (.5,.5);
\coordinate (d) at (-.5,.5);
\fill (a)  circle (2pt);
\fill (b) circle (2pt);
\fill (c)  circle (2pt);
\fill (d) circle (2pt);
\fill[fill=white] (a)  -- (b) -- (c) -- (d) -- cycle;
\draw (a)  -- (b)coordinate[pos=.5](e) -- (c)coordinate[pos=.5](f) -- (d)coordinate[pos=.5](g) -- (a)coordinate[pos=.5](h);

\node[below] at (e) {$2$};
\node[above] at (g) {$3$};
\node[right] at (f) {$1$};
\node[left] at (h) {$1$};

 \draw[postaction={decorate},red] (-.4,.5) .. controls ++(90:.6)  and ++(90:.5) .. (-.8,0)  .. controls         ++(-90:.5) and ++(-90:.6) .. (-.4,-.5);
 \node[red] at (-.97,0) {$\beta$};
  \draw[postaction={decorate},blue] (-.5,.35) .. controls ++(180:.7)  and ++(180:.5) .. (0,1)  .. controls         ++(0:.5) and ++(0:.7) .. (.5,.35);
   \node[blue] at (.05,1.1) {$\alpha$};
\end{scope}

%pole2,2
\begin{scope}[xshift=3.7cm,yshift=0cm]
\fill[fill=black!10] (0,0) coordinate (Q) circle (1.1cm);
\coordinate (a) at (-.5,0);
\coordinate (b) at (.5,0);
\fill (a)  circle (2pt);
\fill (b) circle (2pt);
\draw (a) -- (b)coordinate[pos=.5](d);

\node[above] at (d) {$2$};
\node[below] at (d) {$3$};

 \draw[postaction={decorate},red] (-.4,0) .. controls ++(90:.5)  and ++(90:.5) .. (-.7,0)  .. controls         ++(-90:.5) and ++(-90:.5) .. (-.4,0);
\end{scope}

\end{tikzpicture}
\caption{Une différentielle de $\omoduli[1](4;-2,-2)$ avec résidus $(0,0)$ dont le nombre de rotation est $1$.} \label{fig:casgeneralgenreunter}
\end{figure}
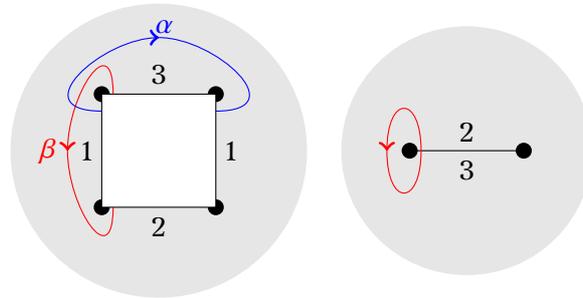

Il reste le cas des composantes connexes de $\omoduli[1](2p;(-2^{p}))$ avec $p$ impair pour lesquelles le nombre de rotation est $2$. On reprend la construction précédente pour les $p-2$ premiers pôles. On associe au pôle $P_{p-1}$ partie polaire est associée aux vecteurs $(i,1;1,i)$. On associe à~$P_{p}$ la partie polaire associée à $(i;i)$. On identifie les bords comme précédemment. Cette construction est illustrée dans la figure~\ref{fig:casgeneralgenreuncuar} dans le cas de la strate $\omoduli[1](6;-2,-2,-2)$  Ainsi, les lacets analogues à ceux de la construction précédente auront  pour indices respectifs~$p-1$ et~$2$. Le nombre de rotation de la surface est donc~$2$.
\begin{figure}[htb]
 \center
\begin{tikzpicture}[scale=1.5, decoration={    markings,
    mark=at position 0.5 with {\arrow[very thick]{>}}}]
%pole2,1
\begin{scope}[xshift=1cm,yshift=0cm]
\fill[fill=black!10] (0,0) coordinate (Q) circle (1.3cm);
\coordinate (a) at (-.5,-.5);
\coordinate (b) at (.5,-.5);
\coordinate (c) at (.5,.5);
\coordinate (d) at (-.5,.5);
\fill (a)  circle (2pt);
\fill (b) circle (2pt);
\fill (c)  circle (2pt);
\fill (d) circle (2pt);
\fill[fill=white] (a)  -- (b) -- (c) -- (d) -- cycle;
\draw (a)  -- (b)coordinate[pos=.5](e) -- (c)coordinate[pos=.5](f) -- (d)coordinate[pos=.5](g) -- (a)coordinate[pos=.5](h);

\node[below] at (e) {$2$};
\node[above] at (g) {$3$};
\node[right] at (f) {$b$};
\node[left] at (h) {$a$};

 \draw[postaction={decorate},red] (-.4,.5) .. controls ++(90:.6)  and ++(90:.5) .. (-.8,0)  .. controls         ++(-90:.5) and ++(-90:.6) .. (-.4,-.5);
 \node[red] at (-.97,0) {$\beta$};
  \draw[postaction={decorate},blue] (-.5,.35) .. controls ++(180:.7)  and ++(180:.5) .. (0,1)  .. controls         ++(0:.5) and ++(0:.7) .. (.5,.35);
   \node[blue] at (.05,1.1) {$\alpha$};
\end{scope}

%pole2,2
\begin{scope}[xshift=3.7cm,yshift=0cm]
\fill[fill=black!10] (0,0) coordinate (Q) circle (1.1cm);
\coordinate (a) at (-.5,0);
\coordinate (b) at (.5,0);
\fill (a)  circle (2pt);
\fill (b) circle (2pt);
\draw (a) -- (b)coordinate[pos=.5](d);

\node[above] at (d) {$2$};
\node[below] at (d) {$3$};

 \draw[postaction={decorate},red] (-.4,0) .. controls ++(90:.5)  and ++(90:.5) .. (-.7,0)  .. controls         ++(-90:.5) and ++(-90:.5) .. (-.4,0);
\end{scope}

%pole2,3
\begin{scope}[xshift=-1.7cm,yshift=0cm]
\fill[fill=black!10] (0,0) coordinate (Q) circle (1.1cm);
\coordinate (a) at (0,-.5);
\coordinate (b) at (0,.5);
\fill (a)  circle (2pt);
\fill (b) circle (2pt);
\draw (a) -- (b)coordinate[pos=.5](d);

\node[left] at (d) {$b$};
\node[right] at (d) {$a$};

 \draw[postaction={decorate},blue] (0,.4) .. controls ++(180:.5)  and ++(180:.5) .. (0,.7)  .. controls         ++(0:.5) and ++(0:.5) .. (0,.4);
\end{scope}
\end{tikzpicture}
\caption{Une différentielle de $\omoduli[1](6;-2,-2,-2)$ avec résidus $(0,0,0)$ dont le nombre de rotation est $2$.} \label{fig:casgeneralgenreuncuar}
\end{figure}
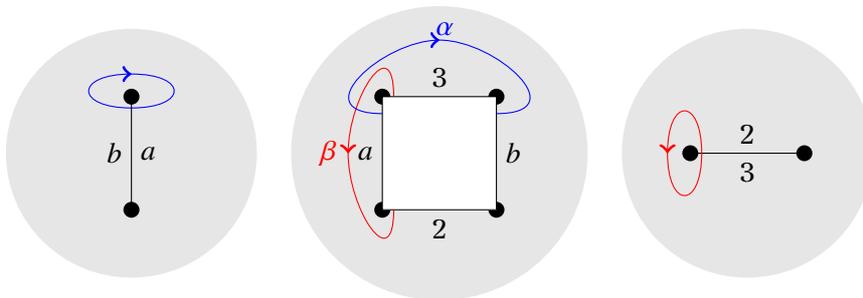
\end{proof}

\subsection{Différentielles en genre supérieur ou égal à $2$}
\label{sec:ggeq2}

Les composantes connexes des strates ont été classifiées par Boissy dans le théorème 1.2 de \cite{Bo}. La preuve du théorème~\ref{thm:ggeq1abel} dans le cas $g\geq2$ n'utilise pas la description des composantes connexes mais repose sur le fait que chaque composante connexe peut s'obtenir via les opérations de couture d'anse et éclatement d'un zéro.

\begin{proof}[Preuve du théorème~\ref{thm:ggeq1abel}  en genre $g\geq2$]
Commençons par le cas des strates avec un unique zéro. Plus précisément, fixons une partition $\mu:=(a;-b_{1},\dots,-b_{p};(-1^{s}))$ de $2g-2$ avec $g\geq2$ et un uplet $r$ dans l'espace résiduel $\espres(\mu)$. Par le résultat de la section~\ref{sec:g1}, chaque composante connexe de la strate $\omoduli[1](a-2g;-b_{1},\dots,-b_{p};(-1^{s}))$ contient une différentielle $\omega$ dont les résidus sont $r$.  De plus, la proposition~6.1 de \cite{Bo} montre que chaque composante connexe d'une strate de genre $g \geq 2$ peut être obtenue par l'ajout d'une anse à un zéro d'une différentielle de genre $g-1$. La proposition~\ref{prop:attachanse} montre que la couture d'anse ne modifie pas les résidus. On obtient donc une différentielle dans la strate $\omoduli[g](\mu)$ dont les résidus sont~$r$ en cousant successivement $g-1$ anses à l'unique zéro de~$\omega$.

Nous considérons maintenant le cas des strates avec $n\geq 2$ zéros. Fixons une partition quelconque $\mu:=(a_{1},\ldots,a_{n};-b_{1},\dots,-b_{p};(-1^{s}))$ de $2g-2$ avec $g\geq2$. 
% Le paragraphe précédent montre que l'application résiduelle restreinte à chaque composante connexe de  la strate $\omoduli[g](\sum a_{i};-b_{1},\dots,-b_{p};(-1^{s}))$ est surjective. 
La proposition~7.1 de \cite{Bo} implique que l'éclatement du zéro de différentielles $\omoduli[g](\sum a_{i};-b_{1},\dots,-b_{p};(-1^{s}))$ permet d'obtenir une différentielle dans chaque composante connexe de $\omoduli(\mu)$. La proposition~\ref{prop:eclatZero} montre que cette opération ne modifie pas les résidus aux pôles. La surjectivité de la restriction  de l'application résiduelle $\appres(\mu)$ à chaque composante connexe est donc induite par la surjectivité de l'application résiduelle sur chaque composantes connexes de la strate $\omoduli[g](\sum a_{i};-b_{1},\dots,-b_{p};(-1^{s}))$ que nous avons montré au paragraphe précédent.
\end{proof}

% 
%  Pour terminer cette section, nous allons montrer que ne résultats impliquent l'existence de différentielles abéliennes holomorphes pour toutes les partitions de $2g-2$. Ce résultat à été montré pour la première fois dans \cite{masm}.
%  \begin{prop}
%   Soit $\mu$ une partition positive de $2g-2$. Il existe une différentielle abélienne holomorphe dans la strate $\omoduli(2g-2)$.
%  \end{prop}
% 
%  \begin{proof}
%   Considérons la différentielle entrelacée $(X,\omega)$ 
%  \end{proof}

\section{Applications}
\label{sec:appli}

Cette section est dédiée à la preuve des applications énoncées dans la section~\ref{sec:apliintro}. La section~\ref{sec:limwei} donne la preuve de la proposition~\ref{prop:limWei} sur les points de Weierstra\ss~et la section~\ref{sec:cylindres} la preuve de la proposition~\ref{prop:cylindres} sur les cylindres dans les surfaces de translation.

\smallskip
\par
\subsection{Limites des points de Weierstra\ss}
\label{sec:limwei}

L'étude des limites des points de \Weierstrass~dans la compactification de Deligne-Mumford a pris une grande ampleur grâce aux travaux d'Eisenbud et Harris sur les séries linéaires limites (voir e.g. \cite{eiha}). L'une des limites de leur méthode est de se restreindre aux courbes de type compact. Esteves et Medeiros l'ont étendue aux courbes ayant deux composantes dans \cite{esme}. Nos résultats et la description de la compactification de la variété d'incidence de \cite{BCGGM} permettent en théorie une description complète des limites de points de \Weierstrass~dans~$\barmoduli[g,1]$.  Nous illustrons cela dans les cas de la proposition~\ref{prop:limWei}.
\smallskip
\par
Nous montrons tout d'abord que l'adhérence du lieu de \Weierstrass~ne rencontre pas celui des courbes stables $X$ où $g$ courbes elliptiques sont attachées à un $\PP^{1}$ contenant le point marqué (voir le théorème 3.1 de \cite{eiha}). Ces courbes sont représentées à gauche de la figure~\ref{fig:courbesWei}.

 \begin{figure}[htb]
\begin{tikzpicture}[scale=1.8]
%premiere Curve
\begin{scope}[xshift=-2.5cm]
\draw (-4,0.1) coordinate (x1) .. controls (-3.8,-.3) and (-4.2,-.6) .. (-4,-1);
\draw (-3.4,0.1) coordinate (x4) .. controls (-3.2,-.3) and (-3.6,-.6) .. (-3.4,-1);
\draw (-2,0.1)  coordinate (x2).. controls (-1.8,-.3) and (-2.2,-.6) .. (-2,-1);

\draw (-4.2,-.8) .. controls (-3.3,-1) and (-2.6,-.6) .. (-1.8,-.8) coordinate (x3)
coordinate [pos=.6] (z);
\fill (z) circle (1pt);\node [below] at (z) {$z$};
\node [above] at (x1) {$X_{1}$};\node [above] at (x2) {$X_{g}$};\node [right] at (x3) {$\PP^{1}$};\node [above] at (x4) {$X_{2}$};
\node at (-2.7,-.3) {$\dots$};
\end{scope}

%seconde Curve
\begin{scope}[xshift=2.5cm]
\draw (-4.8,-1) .. controls (-4.6,.5) and (-4.2,.5) .. (-4,-1) coordinate[pos=.5] (x1);
\draw (-3.4,0.1) coordinate (x4) .. controls (-3.2,-.3) and (-3.6,-.6) .. (-3.4,-1);
\draw (-2,0.1)  coordinate (x2).. controls (-1.8,-.3) and (-2.2,-.6) .. (-2,-1);

\draw (-5,-.8) .. controls (-3.7,-1) and (-2.9,-.6) .. (-1.8,-.8) coordinate (x3)
coordinate [pos=.7] (z);
\fill (z) circle (1pt);\node [below] at (z) {$z$};
\node [above] at (x1) {$X_{0}$};\node [above] at (x2) {$X_{g-2}$};\node [right] at (x3) {$\PP^{1}$};\node [above] at (x4) {$X_{1}$};
\node at (-2.7,-.3) {$\dots$};
\end{scope}
\end{tikzpicture}
\caption{Les courbes pointées considérées dans la  proposition~\ref{prop:limWei}. Les courbes $X_{i}$ sont de genre~$1$.}
\label{fig:courbesWei}
\end{figure}
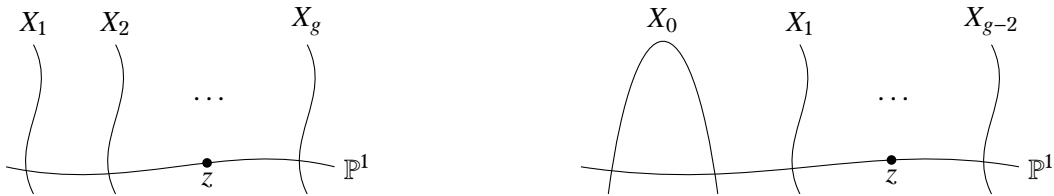

Par \cite[Section~3.6]{BCGGM}, l'adhérence du lieu de \Weierstrass~dans $\barmoduli[g,1]$ coïncide avec la projection de la compactification de la variété d'incidence de $\omoduli[g](g,1,\dots,1)$. D'après le théorème~1.3 de l'article loc. cit., il suffit de montrer qu'il n'existe pas de différentielle entrelacée lissable sur une courbe semi-stablement équivalente à~$X$.

Soit~$\xi$ une différentielle entrelacée sur (une courbe semi-stablement équivalente à)~$X$. La restriction~$\xi_{0}$ de $\xi$ à $\PP^{1}$ possède un zéro d'ordre supérieur ou égal à $g$. De plus,  $\xi_{0}$ possède une singularité d'ordre supérieur ou égal à $-2$ (et distinct de $-1$) aux points nodaux de~$\PP^{1}$. Notons~$-b_{i}$ l'ordre des pôles aux points nodaux. Comme la différentielle $\xi_{0}$ possède au plus $g$ pôles et que ceux-ci sont d'ordre $-2$, on obtient l'inégalité $\sum (b_{i}-1) \leq g$. Donc la différentielle~$\xi_{0}$ sur $\PP^{1}$ vérifie l'inégalité~\eqref{eq:genrezeroresiduzerofr}. Le théorème~\ref{thm:geq0keq1} implique qu'au moins deux pôles possèdent des résidus non nuls. La {\em condition résiduelle globale} de la définition 1.2 de \cite{BCGGM} n'est donc pas satisfaite et~$\xi$ n'est pas lissable.

Maintenant, nous montrons que l'adhérence du lieu de \Weierstrass~coupe le lieu de $\barmoduli[g,1]$ donné de la façon suivante. Ces courbes sont formées d'un $\PP^{1}$ contenant le point marqué attaché à $g-2$ courbes elliptiques $X_{1},\dots,X_{g-2}$ par un unique point et à une courbe elliptique~$X_{0}$ par deux points. Elles sont représentées à droite de la figure~\ref{fig:courbesWei}.

Il suffit construire une différentielle entrelacée lissable de type $(g,(1^{g-2}))$ sur une de ces courbes. Sur toutes les courbes elliptiques, on prend la différentielle holomorphe. Sur la courbe projective, on prend une différentielle dans $\omoduli[0](g,(1^{g-2});(-2^{g}))$ avec $g-2$ résidus nuls. Une telle différentielle existe par le théorème~\ref{thm:geq0keq1}. On colle alors la courbe elliptique~$X_{0}$ aux deux pôles dont les résidus ne sont pas nuls. Les autres courbes elliptiques sont collées aux pôles dont les résidus sont nuls. Le théorème 1.3 de \cite{BCGGM} implique que cette différentielle entrelacée est lissable.

\smallskip
\par
\subsection{Cylindres dans une surface plate}
\label{sec:cylindres}
Un {\em cylindre} dans une surface de translation est un cylindre composé de géodésiques fermées parallèles et dont le bord est formé de connexions de selles.
Naveh a montré dans \cite{Na} que le nombre maximal de cylindres disjoints dans une surface de la strate $S:=\omoduli(a_{1},\dots,a_{n})$ est $g+n-1$.
\smallskip
\par
Nous décrivons les périodes possibles des circonférences de ces cylindres, prouvant la proposition~\ref{prop:cylindres}. On fixe $(\lambda_{1},\dots,\lambda_{t})\in (\mathbb{C}^{\ast}/\{\pm1\})^{t}$ pour le reste de cette section.
\par
Supposons qu'il existe une différentielle $\omega$ de $S$ qui possède une famille de~$t$ cylindres disjoints de circonférences respectives $\lambda_{1},\dots,\lambda_{t}$. Nous montrons l'existence d'une différentielle stable dont les zéros sont d'ordres $(a_{1},\dots,a_{n})$ avec des pôles simples aux nœuds dont les résidus sont $\pm \lambda_{i}$.  Coupons $\omega$ le long d'une d'une géodésique périodique dans chaque cylindre. Nous pouvons alors remplacer les deux demi-cylindres obtenus par deux demi-cylindres infinis de même circonférence. On obtient donc une différentielle entrelacée avec des pôles simples aux nœuds. De plus, les résidus de ces pôles sont égaux à plus ou moins la circonférence des cylindres. On en déduit le sens direct de la proposition~\ref{prop:cylindres}.
\par
La direction réciproque est une application directe du lissage des nœuds des différentielles stables avec pôles simples aux nœuds (voir le lemme~\ref{lem:lisspolessimples}).
\smallskip
\par
Ce résultat permet de comprendre parfaitement les circonférences des cylindres sur une surface de translation pour $t <g$ et dans le cas minimal pour $t=g$.
\begin{cor}\label{cor:cylindres}
% Le $t$-uplet $(\lambda_{1},\dots,\lambda_{t})\in (\mathbb{C}^{\ast})^{t}$ est un vecteur constitué des périodes des circonférences de cylindres disjoints d'une différentielle de $\omoduli(2g-2)$ si et seulement si le $2t$-uplet $(\lambda_{1},\dots,\lambda_{t},-\lambda_{1},\dots,-\lambda_{t})$ est dans l'image de  $\appres[g-t](2g-2;(-1^{2t}))$.
% 
 Tous les $t$-uplets de $\CC^{\ast}/\{\pm1\}$ avec $t < g$ sont  réalisables comme périodes des circonférences de cylindres disjoints d'une différentielle de $\omoduli(a_{1},\dots,a_{n})$.
 
 Un $g$-uplet de $\CC^{\ast}/\{\pm1\}$ est  constitué des périodes des circonférences de cylindres disjoints d'une différentielle de $\omoduli(2g-2)$ si et seulement si il n'est pas de la forme $c \cdot (\lambda_{1},\dots, \lambda_{g})$ avec $c$ dans $\CC^{\ast}$ et $(\lambda_{1},\dots,\lambda_{g})\in \NN$ tel que  $\sum_{i=1}^{g} \lambda_{i} \leq 2g-2$.
\end{cor}

Avant de passer à la preuve de ce résultat nous faisons une remarque sur le cas $t=g$ avec $n\geq2$ zéros.
\begin{rem}
 On considère une strate $\omoduli(a_{1},\dots,a_{n})$.  En général, il peut exister une différentielle qui possède des cylindres disjoints de circonférences $(\lambda_{1},\dots,\lambda_{g})\in \NN$ avec $\sum_{i=1}^{g} \lambda_{i} \leq \max (a_{1},\dots,a_{n})$. Par exemple, il existe une différentielle avec des cylindres disjoints de circonférences $(1,1,1,1)$ dans la strate $\omoduli[4](4,1,1)$. Une telle différentielle peut être obtenue en lissant la différentielle stable suivante. Elle est formée à partir de deux différentielles. La première est une différentielle de la strate $\omoduli[0](1,1;-1,-1,-1,-1)$ dont les résidus sont égaux à $(1,1,-1,-1)$. La seconde est une différentielle de $\omoduli[1](4;-1,-1,-1,-1)$ dont les résidus sont $(1,1,-1,-1)$. Pour chaque différentielle, on colle deux pôles simples de résidus opposés entre eux. Puis on colle les deux pôles simples restant d'une différentielle aux deux pôles simples restant de l'autre différentielle.
\end{rem}

\begin{proof}
 Si $t<g$, il existe une différentielle stable sur une courbe stable irréductible qui satisfait les hypothèses de la proposition~\ref{prop:cylindres}. En effet,  le théorème~\ref{thm:ggeq1abel} donne une différentielle $\omega$ dans la strate $\omoduli[g-t](a_{1},\dots,a_{n};(-1^{2t}))$ dont les résidus sont $\pm \lambda_{i}$. On forme la différentielle stable de la proposition~\ref{prop:cylindres} en collant deux pôles simples dont les résidus sont opposés.   
 \smallskip
 \par
 On considère maintenant le cas où $t=g$ et dans la strate $\omoduli(2g-2)$.  Supposons que le $g$-uplet n'est pas de la forme $c \cdot (\lambda_{1},\dots, \lambda_{g})$ avec $c$ dans $\CC^{\ast}$ et $(\lambda_{1},\dots,\lambda_{g})\in \ZZ$ tel que  $\sum_{i=1}^{g} \lambda_{i} \leq 2g-2$.  Le point ii) du théorème~\ref{thm:geq0keq1} donne l'existence d'une différentielle  dans $\omoduli[0](2g-2;(-1^{2g}))$ dont les résidus sont $\pm \lambda_{i}$. On forme la différentielle stable de la proposition~\ref{prop:cylindres} en collant deux pôles simples dont les résidus sont opposés.  Réciproquement, supposons que le $g$-uplet n'est pas de la forme donné par le corollaire~\ref{cor:cylindres}. Notons tout d'abord que toutes les différentielles stables de la proposition~\ref{prop:cylindres} sont irréductibles dans ce cas. En effet, chaque composante contient au moins un zéro de la différentielle. En normalisant la courbe stable sous jacente on obtient une différentielle de $\omoduli[0](2g-2;(-1^{2g}))$ dont les résidus sont $\pm \lambda_{i}$. Cela contredit le point ii) du théorème~\ref{thm:geq0keq1}.
\end{proof}

\smallskip
\par
Pour terminer, nous montrons la proposition~\ref{prop:cylcc} qui affirme  l'existence une différentielle avec $g+n-1$ cylindres disjoints dans chaque composante connexe des strates de différentielles holomorphes. La classification des composantes connexes est donnée dans les théorèmes 1 et 2 de \cite{kozo}. Nous rappelons simplement ici que les composantes connexes sont classifiées par hyperellipticité et parité de la structure spin. 

La preuve de la proposition~\ref{prop:cylcc} suit le schéma suivant. Dans le cas des strates minimales, i.e. avec $n=1$, nous construisons une différentielle stable sur une courbe algébrique irréductible de genre géométrique~$0$ avec $g$ nœuds dont le lissage par le lemme~\ref{lem:lisspolessimples} donne une différentielle dans la composante souhaitée. Chaque nœud correspond à un cylindre sur la différentielle lissée. De cette manière, on obtient la borne souhaitée dans toutes les composantes de la strate minimale. Les strates non minimales se traitent en dégénérant les différentielles stables du cas minimal, puis en lissant ces différentielles.
% Comme les strates sont connexes pour $g\leq2$, nous supposerons par la suite que $g\geq3$. 

 Commençons par la strate $\omoduli[3](4)$ qui possède deux composantes: l'une hyperelliptique et l'autre impaire. Considérons trois nombres réels strictement positifs $r_{1},r_{2},r_{3}$ non commensurables entre eux et les parties polaires d'ordre $1$ associées à $(\pm r_{i};\emptyset)$. Collons successivement sur un segment  les cylindres associés à $r_{i}$  avec $i$ croissant. Sur la partie inférieure du segment, on colle successivement les cylindres associés à $-r_{3}$, $-r _{2}$ et $-r_{1}$ dans le cas hyperelliptique et $-r_{2}$, $-r_{3}$ et $-r_{1}$ dans le cas impair. Les différentielles possédant les propriétés souhaitées sont obtenues en lissant les différentielles stables construites en collant les points à l'infini des cylindres de même circonférence. 
% 
% 
% Montrons tout d'abord qu'il existe une différentielle abélienne avec $g$ cylindres dans la composante hyperelliptique de $\omoduli(2g-2)$. Considérons le polygone régulier $\mathcal{P}_{2g}$ avec $2g$ arêtes. Nous collons un cylindre semi-infini à chaque arête de $\mathcal{P}_{2g}$. Nous formons une différentielle stable en collant les points à l'infini de deux cylindres opposés. Cette différentielle est hyperelliptique et l'involution hyperelliptique est la rotation d'angle $\pi$ de centre le centre $\mathcal{P}_{2g}$. On obtient donc une différentielle abélienne avec $g$ cylindres dans la composante hyperelliptique de $\omoduli(2g-2)$ en lissant les nœuds grâce au lemme~\ref{lem:lisspolessimples}. 
% 
% Le cas de la composante hyperelliptique de la strate $\omoduli(g-1,g-1)$ s'obtient de manière similaire en partant du polygone  régulier $\mathcal{P}_{2g+2}$. On identifie par translation deux arêtes opposées et on colle des demi-cylindres infinis aux autres arêtes. Le reste de la preuve est similaire au cas minimal.

Dans le cas $g\geq4$ nous construisons une différentielle  qui contient~$g$ cylindres dans chacune des trois composantes de $\omoduli(2g-2)$. Dans la suite, nous décrirons uniquement les normalisations des différentielles stables dont le lissage possède ces propriétés. Considérons~$g$ nombres réels strictement positifs $r_{i}$ non commensurables entre eux. Prenons les~$2g$ parties polaires d'ordre~$1$ associées respectivement à $(\pm r_{i};\emptyset)$. On colle successivement sur un segment  les $g$ cylindres positifs associés à~$r_{i}$  avec $i$ croissant. Sur la partie inférieure de ce segment on colle successivement les cylindres associés à $-r_{i}$ avec $i$ décroissant dans le cas hyperelliptique. Pour les composantes de parité, nous collons les cylindres associés à $-r_{i}$ dans les ordres $g,1,2,\dots,g-1$ et $g,1,2,\dots,g-3,g-1,g-2$. Ces deux constructions sont illustrées dans la figure~\ref{fig:compconcyl}.

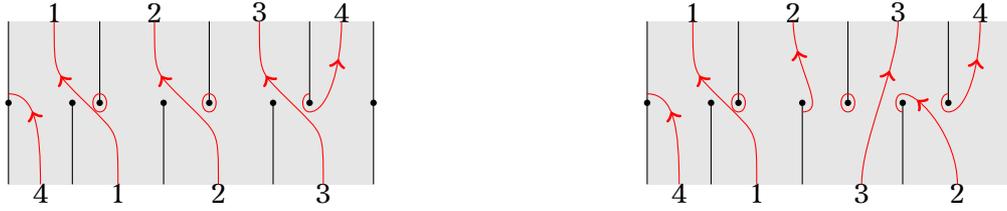
\begin{figure}[htb]
\center
\begin{tikzpicture}[scale=1.2, decoration={    markings,
    mark=at position 0.7 with {\arrow[very thick]{>}}}]
    
    \begin{scope}[xshift=-4cm]
%Curve
\fill[fill=black!10] (-2,-.9)  -- (2,-.9)-- (2,.9)-- (-2,.9) -- cycle;

\coordinate (a0) at (-2,0);
\coordinate (a1) at (-1,0);
\coordinate (a2) at (.2,0);
\coordinate (a3) at (1.3,0);
\coordinate (a4) at (2,0);
\coordinate (b0) at (-2,0);
\coordinate (b1) at (-1.3,0);
\coordinate (b2) at (-.3,0);
\coordinate (b3) at (.9,0);
\coordinate (b4) at (2,0);

\foreach \i in {0,...,4}
\fill (a\i)  circle (1pt);
   \foreach \i in {0,...,4}
  \draw (a\i) -- ++(0,.9);
   \foreach \i in {0,...,4}
  \draw (b\i) -- ++(0,-.9);
   \foreach \i in {1,2,3}
   \fill (b\i)  circle (1pt);
     
     \node at (-1.5,1) {$1$};
     \node at (-.4,1) {$2$};
     \node at (.75,1) {$3$};
     \node at (1.65,1) {$4$};
       \node at (-.8,-1) {$1$};
     \node at (.3,-1) {$2$};
     \node at (1.45,-1) {$3$};
     \node at (-1.65,-1) {$4$};
     
 \draw[postaction={decorate},red] (-.8,-.9) .. controls ++(90:.6)  and ++(-45:.5) .. (-1.15,0)  .. controls         ++(135:.5) and ++(-90:.6) .. (-1.5,.9);    
\draw[postaction={decorate},red] (.3,-.9) .. controls ++(90:.6)  and ++(-45:.5) .. (-.05,0)  .. controls         ++(135:.5) and ++(-90:.6) .. (-.4,.9);
\draw[postaction={decorate},red] (1.45,-.9) .. controls ++(90:.6)  and ++(-45:.5) .. (1.1,0)  .. controls         ++(135:.5) and ++(-90:.6) .. (.75,.9);

 \draw[postaction={decorate},red] (-1.65,-.9) .. controls ++(90:.6)  and ++(0:.3) .. (-2,.1);    
  \draw[red] (-1,.1) .. controls ++(180:.1)  and ++(180:.1) .. (-1,-.1)  .. controls         ++(0:.1) and ++(0:.1) .. (-1,.1);
  \draw[red] (.2,.1) .. controls ++(180:.1)  and ++(180:.1) .. (.2,-.1)  .. controls         ++(0:.1) and ++(0:.1) .. (.2,.1);
    \draw[postaction={decorate},red] (1.3,.1) .. controls ++(180:.1)  and ++(180:.1) .. (1.3,-.1)  .. controls         ++(0:.2) and ++(-90:.5) .. (1.65,.9);
\end{scope}

  \begin{scope}[xshift=3cm]
%Curve 2
\fill[fill=black!10] (-2,-.9)  -- (2,-.9)-- (2,.9)-- (-2,.9) -- cycle;

\coordinate (a0) at (-2,0);
\coordinate (a1) at (-1,0);
\coordinate (a2) at (.2,0);
\coordinate (a3) at (1.3,0);
\coordinate (a4) at (2,0);
\coordinate (b0) at (-2,0);
\coordinate (b1) at (-1.3,0);
\coordinate (b2) at (-.3,0);
\coordinate (b3) at (.8,0);
\coordinate (b4) at (2,0);

\foreach \i in {0,...,4}
\fill (a\i)  circle (1pt);
   \foreach \i in {0,...,4}
  \draw (a\i) -- ++(0,.9);
   \foreach \i in {0,...,4}
  \draw (b\i) -- ++(0,-.9);
   \foreach \i in {1,2,3}
   \fill (b\i)  circle (1pt);
     
     \node at (-1.5,1) {$1$};
     \node at (-.4,1) {$2$};
     \node at (.75,1) {$3$};
     \node at (1.65,1) {$4$};
       \node at (-.8,-1) {$1$};
     \node at (.35,-1) {$3$};
     \node at (1.4,-1) {$2$};
     \node at (-1.65,-1) {$4$};
     
 \draw[postaction={decorate},red] (-.8,-.9) .. controls ++(90:.6)  and ++(-45:.5) .. (-1.15,0)  .. controls         ++(135:.5) and ++(-90:.6) .. (-1.5,.9);    
\draw[postaction={decorate},red] (.35,-.9) .. controls ++(90:.5)  and ++(-90:.5) .. (.75,.9);

\draw[postaction={decorate},red] (1.4,-.9) .. controls ++(90:.5)  and ++(0:.2) .. (.8,.1)  .. controls         ++(180:.1) and ++(180:.1) .. (.8,-.1);
\draw[postaction={decorate},red] (-.3,-.1) .. controls ++(0:.3)   and ++(-90:.5) .. (-.4,.9);

 \draw[postaction={decorate},red] (-1.65,-.9) .. controls ++(90:.6)  and ++(0:.3) .. (-2,.1);    
  \draw[red] (-1,.1) .. controls ++(180:.1)  and ++(180:.1) .. (-1,-.1)  .. controls         ++(0:.1) and ++(0:.1) .. (-1,.1);
  \draw[red] (.2,.1) .. controls ++(180:.1)  and ++(180:.1) .. (.2,-.1)  .. controls         ++(0:.1) and ++(0:.1) .. (.2,.1);
    \draw[postaction={decorate},red] (1.3,.1) .. controls ++(180:.1)  and ++(180:.1) .. (1.3,-.1)  .. controls         ++(0:.2) and ++(-90:.5) .. (1.65,.9);
\end{scope}

\end{tikzpicture}
\caption{Différentielles stables dans le bord des composantes impaires et paires de la composante $\omoduli[4](6)$.} \label{fig:compconcyl}
\end{figure}

Les deux différentielles ainsi formées sont clairement non hyperelliptiques. Il reste donc à vérifier qu'elles sont de parités distinctes. Nous décrivons des lacets qui forment une base symplectique de l'homologie sur la différentielle lissée. Les $g$ premiers lacets $\alpha_{i}$ sont donnés par une géodésique périodique dans chaque cylindre associé à $r_{i}$. Les $g$ autres lacets~$\beta_{i}$ sont des courbes qui vont du cylindre de résidu $-r_{i}$ à celui de résidu $r_{i}$ en intersectant uniquement~$\alpha_{i}$.  Un choix possible pour les $\beta_{i}$ est représenté sur la figure~\ref{fig:compconcyl}.  Les indices des~$\alpha_{i}$ sont nuls. De plus, dans la première différentielle, on peut choisir pour $i\in\{1,\dots,g-1\}$ les $\beta_{i}$ tels qu'ils n'intersectent pas les demi-droites verticales représentées sur la figure~\ref{fig:compconcyl}. Dans la seconde différentielle, c'est possible pour $i\in\{ 1,\dots,g-3,g-1 \}$. Tous ces lacets sont d'indice nul. Dans les deux différentielles, pour joindre les cylindres d'indice $g$, on peut choisir un lacet d'indice $g-1$ comme représenté sur la figure~\ref{fig:compconcyl}. Enfin dans la seconde différentielle, on peut choisir $\beta_{g-2}$ d'indice $1$ comme représenté sur la figure~\ref{fig:compconcyl}. Donc les indices des courbes de ces deux familles diffèrent uniquement pour le lacet $\beta_{g-2}$. L'équation~(4) de \cite{kozo} de la parité d'une différentielle en fonction des indices des courbes d'une base symplectique implique que ces deux différentielles sont de parités distinctes.

Enfin nous montrons  qu'il existe une différentielle qui possède $g+n-1$ cylindres dans chaque composante connexe des strates avec $n\geq2$ zéros. Pour cela, nous dégénérons les différentielles stables du cas minimal de la façon suivante. Soit $a_{1}$ un zéro d'ordre inférieur ou égal à $g-1$ et maximal pour cette propriété. Nous choisissons sur les différentielles précédentes $a_{1}+1$ pôles consécutifs dont les résidus sont de même signe. Puis nous ajoutons un cylindre vertical fini entre ces cylindres et les autres cylindres de la différentielle. Enfin nous faisons tendre la hauteur de ce cylindre vers l'infini. Les graphes duaux de différentielles stables obtenues de manière similaire sont présentées dans la figure~\ref{fig:graphdeliant}. On considère maintenant le zéro d'ordre $a_{2}\leq a_{1}$ et maximal pour cette propriété. On fait la même construction sur la composante irréductible de la différentielle stable qui contient le zéro d'ordre $2g-2-a_{1}$. En effet, cette composante possède au moins $\tfrac{2g-2-a_{1}}{2} \geq a_{2}$ pôles dont les résidus sont de même signes. En itérant cette construction on obtient une différentielle stable au bord de chaque  composante connexe. De plus, elle préserve la parité (et l'hyperellipticité dans le cas $\omoduli(g-1,g-1)$)  des différentielles stables.  Comme cette opération ajoute un cylindre pour chaque zéro, la différentielle obtenue en lissant ces différentielles stables possède $g+n-1$ cylindres et appartient à la composante souhaitée.

\printbibliography

\end{document}